\documentclass[10pt, reqno]{amsart}
\usepackage{amsmath, amssymb, amsthm, fancyhdr, verbatim, graphicx}
\usepackage{enumerate}
\usepackage[all]{xy}
\usepackage[T1]{fontenc} 
\usepackage{framed, hyperref}
\usepackage[OT2, T1]{fontenc}
\usepackage[titletoc]{appendix}
\usepackage{etex}
\usepackage[usenames,dvipsnames]{xcolor}
\usepackage{mathrsfs}
\usepackage{tikz-cd}

\numberwithin{equation}{subsection}

\DeclareSymbolFont{cyrletters}{OT2}{wncyr}{m}{n}
\DeclareMathSymbol{\Sha}{\mathalpha}{cyrletters}{"58}

\usepackage{color}

\newcommand{\F}{\mathbf{F}}

\newcommand{\G}{\mathbf{G}}

\newcommand{\wt}[1]{\widetilde{#1}}

\newcommand{\Q}{\mathbf{Q}}
\newcommand{\Z}{\mathbf{Z}}

\newcommand{\Gal}{\operatorname{Gal}}

\newcommand{\R}{\mathbf{R}}

\newcommand{\ul}[1]{\underline{#1}}
\newcommand{\ol}[1]{\overline{#1}}
\newcommand{\wh}[1]{\widehat{#1}}

\newcommand{\Cal}[1]{\mathcal{#1}}
\newcommand{\A}{\mathbf{A}}
\newcommand{\mbf}[1]{\mathbf{#1}}

\newcommand{\mrm}[1]{\mathrm{#1}}

\newcommand{\co}{\colon}

\DeclareMathOperator{\et}{\acute{e}t}

\DeclareMathOperator{\Ima}{Im\,}
\DeclareMathOperator{\rank}{rank}

\DeclareMathOperator{\Spec}{Spec\,}

\DeclareMathOperator{\End}{End}

\DeclareMathOperator{\Br}{Br}

\DeclareMathOperator{\Pic}{Pic}

\DeclareMathOperator{\Id}{Id}

\DeclareMathOperator{\pt}{pt}

\DeclareMathOperator{\Sq}{Sq}
\DeclareMathOperator{\nd}{nd}
\DeclareMathOperator{\tors}{tors}
\DeclareMathOperator{\odd}{odd}
\DeclareMathOperator{\even}{even}

\DeclareMathOperator{\cla}{cl}
\DeclareMathOperator{\Et}{\acute{E}t}

\DeclareMathOperator{\red}{red}
\DeclareMathOperator{\pr}{pr}

\newtheorem{thm}{Theorem}[section]
\newtheorem{lemma}[thm]{Lemma}
\newtheorem{prop}[thm]{Proposition}
\newtheorem{cor}[thm]{Corollary}

\newtheorem{conj}[thm]{Conjecture}

\theoremstyle{remark}
\newtheorem{remark}[thm]{Remark} 
\newtheorem{defn}[thm]{Definition}
\newtheorem{example}[thm]{Example}

\makeatletter
\def\th@remark{%
  \thm@headfont{\bfseries}%
  \normalfont % body font
}
\def\imod#1{\allowbreak\mkern5mu({\operator@font mod}\,\,#1)}
\makeatother

\title[\'{E}tale Steenrod operations and the Artin-Tate pairing]{ \'{E}tale Steenrod operations \\ and the Artin-Tate pairing }

\author{Tony Feng}

\begin{document}

\begin{abstract}
We prove a 1966 conjecture of Tate concerning the Artin-Tate pairing on the Brauer group of a surface over a finite field, which is the analogue of the Cassels-Tate pairing. Tate asked if this pairing is always alternating and we find an affirmative answer, which is somewhat surprising in view of the work of Poonen-Stoll on the Cassels-Tate pairing. Our method is based on studying a connection between the Artin-Tate pairing and (generalizations of) Steenrod operations in \'{e}tale cohomology. Inspired by an analogy to the algebraic topology of manifolds, we develop tools  allowing us to calculate the relevant \'{e}tale Steenrod operations in terms of characteristic classes.
\end{abstract}

\maketitle

\tableofcontents

\section{Introduction}\label{sec: introduction}
\subsection{Motivation}

Let $X$ be a smooth, projective, geometrically connected surface over $\F_q$, where $\mrm{char} \ \F_q = p$. For every prime $\ell \neq p$, M. Artin and Tate \cite{Tate66} defined a pairing 
\begin{equation}\label{Tate_pairing}
\langle  \cdot, \cdot \rangle_{\mrm{AT}}   \colon \Br(X)_{\nd}[\ell^{\infty}] \times \Br(X)_{\nd}[\ell^{\infty}] \rightarrow \Q/\Z
\end{equation}
where $\Br(X)_{\nd}$ denotes the quotient of the Brauer group $\Br(X)$ by its divisible part, and $\Br(X)_{\nd}[\ell^{\infty}] $ denotes its $\ell$-power torsion subgroup. (Conjecturally the divisible part vanishes, implying that $\Br(X)_{\nd} = \Br(X)$.) We will review the definition of \eqref{Tate_pairing} in \S \ref{subsec: Tate pairing}; we henceforth call it the \emph{Artin-Tate pairing}. 

Artin and Tate's investigation of $\Br(X)$ was motivated by a dictionary relating the invariants of $X$ to those appearing in the Birch and Swinnerton-Dyer Conjecture for abelian varieties over function fields. In particular, under this dictionary $\Br(X)$ corresponds to $\Sha$, and the Artin-Tate pairing corresponds to the Cassels-Tate pairing.

It is not difficult to show that the pairing $\langle  \cdot, \cdot \rangle_{\mrm{AT}} $ is skew-symmetric, but it is much less clear if it is \emph{alternating}. For clarity, we recall that skew-symmetric means that
\[ 
\langle x,y \rangle_{\mrm{AT}} + \langle y,x \rangle_{\mrm{AT}} = 0 \text{ for all } x,y \in \Br(X)_{\nd}[\ell^{\infty}] ,
\]
while alternating means the stronger condition that
\[
\langle x,x \rangle_{\mrm{AT}} = 0 \text{ for all } x \in \Br(X)_{\nd}[\ell^{\infty}] . 
\]
Since the distinction between skew-symmetric and alternating disappears for $\ell \neq 2$, the difficulty lies entirely in the case $\ell = 2$. In Tate's 1966 Bourbaki report on the Artin-Tate Conjecture, he asks \cite[after Theorem 5.1]{Tate66} if the pairing \eqref{Tate_pairing} is alternating, conjecturing that the answer is ``yes''.

\begin{conj}[Tate, 1966] \label{conj: Tate}
The Artin-Tate pairing is alternating. 
\end{conj}

Tate's motivation for making Conjecture \ref{conj: Tate} was Cassels' result \cite{Ca65} that the analogous Cassels-Tate pairing is alternating for elliptic curves, which Tate had generalized in \cite{Tate62} to abelian varieties with principal polarization ``arising from a rational divisor''. Moreover, one can find at the end of \cite[\S 1]{Tate66} the claim that the Cassels-Tate pairing is alternating for \emph{all} Jacobians (with respect to their canonical principal polarizations). But this is rather ironic in hindsight, as Poonen and Stoll eventually demonstrated in \cite{PS99} that the Cassels-Tate pairing actually need \emph{not} be alternating in general for abelian varieties with principal polarization not satisfying the technical condition of ``arising from a rational divisor''; in particular, it need not be alternating for Jacobians. See the introduction and \S 1.2 of \cite{Sher} for a detailed explanation of these technical subtleties and amusing history.
	
The history of Conjecture \ref{conj: Tate} is perhaps even more tortuous than that of the analogous question for the Cassels-Tate pairing. Recall that any finite abelian group with a nondegenerate alternating pairing has order equal to a perfect square\footnote{This should be thought of as analogous to the fact that a vector space with a nondegenerate alternating pairing has even dimension.} \cite[\S 6]{PS99}, so Conjecture \ref{conj: Tate} implies that $\Br(X)_{\nd}[2^{\infty}]$ has square order. Manin computed examples (\cite{Manin67}, \cite{Manin86}) in which $\# \Br(X)_{\nd}[2^{\infty}]$ was purportedly $\Z/2\Z$, seemingly \emph{disproving} Conjecture \ref{conj: Tate}. However, in 1996 Urabe found mistakes in Manin's calculations that invalidated the counterexamples (see the introduction to \cite{Urabe96}), and then proved that in characteristic $p \neq 2$, $\Br(X)_{\nd}[2^{\infty}]$ \emph{always does} have square order! 

%The core observation of \cite{Urabe96} is that the existence of a subgroup on which the pairing is alternating, and which contains its own othogonal complement, is sufficient to deduce that the ambient group has square order. 

There has been some other partial progress on Conjecture \ref{conj: Tate} besides Urabe's theorem. We note in particular the following two results. 
 \begin{itemize}
 \item Zarhin showed in \cite{Zar} that if $X$ lifts to characteristic $0$ and the N\'{e}ron-Severi group of $X \times_{\F_q} \ol{\F}_q$ has vanishing 2-primary part, then $\langle  \cdot, \cdot \rangle_{\mrm{AT}}$ is alternating for $X$.\footnote{We thank Yuri Zarhin for informing us about \cite{Zar}, and for translating the statement of its main theorem into English. Strictly speaking, the result is for a pairing discovered independently in \cite{Zar}, which should be the same as Artin-Tate's, but no comparison with the Artin-Tate pairing is made in \cite{Zar}.}
 \item Liu-Lorenzini-Raynaud \cite{LLR05} proved by that if $\Br(X)$ is finite, implying that $\Br(X) = \Br(X)_{\nd}$, then $\# \Br(X)$ is a perfect square. (We emphasize that this applies even to the $p$-primary part!)\footnote{Technically, the argument of \cite{LLR05} uses an incorrect formula for $\# \Br(X)$, with the error stemming entirely from the false \cite[Lemma 4.2]{Go79}. This formula is corrected in \cite{Geiss}. The correction is by a factor which is a perfect square, and so preserves the conclusion that $\# \Br(X)[\ell^{\infty}]$ is a perfect square. Liu-Lorenzini-Raynaud have prepared a corrigendum \cite{LLRCorr} that corrects or completes several more arguments in \cite{Go79}, and amends the statements of \cite{LLR05} correspondingly.} Amusingly, the argument of \cite{LLR05} has nothing to do with the Artin-Tate pairing, but actually uses the work of Poonen-Stoll quantifying the failure of $\# \Sha$ to be a perfect square. 
 
 \end{itemize}
 	
\subsection{Results} In this paper we answer Tate's question in the \emph{affirmative}, finally bringing closure to this eventful drama.

\begin{thm}\label{thm: proof of conj} Conjecture \ref{conj: Tate} is true. 
\end{thm}

%\begin{remark}
%An earlier version of this paper (\cite[version 1]{Feng17}) proved that $\langle \cdot, \cdot \rangle_{\mrm{AT}}$ is alternating when restricted to $\Br(X)_{\nd}[2]$. The present paper improves the earlier strategy of (\cite[version 1]{Feng17}) to work for $\Br(X)_{\nd}[2^n]$ for all $n$, deducing Theorem \ref{thm: proof of conj} in the limit. It was observed in (\cite[version 1]{Feng17}) that the case $n=1$ already implies, by a formal group-theoretic fact (cf. \cite[Theorem 8]{PS99}), that $\# \Br(X)_{\nd}[2^{\infty}]$ is a perfect square. 
%\end{remark}

In fact, we deduce Theorem \ref{thm: proof of conj} from a more general result that we now describe. In \cite{Jahn15} Jahn defined a generalization of the Artin-Tate pairing to higher Brauer groups. Briefly, if $X$ is a smooth projective variety of even dimension $2d$ over $\F_q$, then its \emph{higher Brauer group} is 
\[
\Br^d(X):= H_L^{2d+1}(X; \Z(d)),
\]
where $H_L$ denotes Lichtenbaum cohomology \cite[\S 2]{Jahn15}. The importance of this group lies in its relation to the (other) Tate Conjecture concerning algebraic cycles in $X$. A completely analogous construction to Artin-Tate's, which we will describe in \S \ref{subsec: Tate pairing}, gives a non-degenerate skew-symmetric pairing for $\ell \neq p$ 
\[
\langle \cdot, \cdot  \rangle_{\mrm{AT}} \co \Br^d(X)_{\nd}[\ell^{\infty}] \times \Br^d(X)_{\nd} [\ell^{\infty}] \rightarrow \Q/\Z.
\]

One wants to know if $\langle \cdot, \cdot  \rangle_{\mrm{AT}}$ is alternating; again the issue is for $\ell=2$. In particular, this would imply that $\# \Br^d(X)_{\nd}[2^{\infty}]$ is a perfect square. Jahn generalized Urabe's method to show that $\# \Br^d(X)_{\nd}[2^{\infty}]$ is indeed a perfect square if $\mrm{char} \ \F_q = p >2$ \cite[Theorem 1]{Jahn15}. We prove the stronger statement that $\langle \cdot, \cdot  \rangle_{\mrm{AT}}$ is alternating for any such $X$, which of course recovers Theorem \ref{thm: proof of conj} when $X$ is specialized to have dimension $2$.

\begin{thm}\label{thm: main}
Let $X$ be a smooth, projective, geometrically connected variety of dimension $2d$ over $\F_q$ with $\mrm{char} \ \F_q = p \neq 2$. The pairing $\langle  \cdot, \cdot \rangle_{\mrm{AT}} $ on $\Br^d(X)_{\nd}[2^{\infty}]$ is alternating. 
\end{thm}

In the course of proving Theorem \ref{thm: main}, we establish several results which may be of independent interest and utility, as our work involves developing algebro-geometric versions of techniques of fundamental importance in algebraic topology. Let us briefly summarize the idea, which will be elaborated upon in \S \ref{subsec: overview}. The skew-symmetry of $\langle  \cdot, \cdot \rangle_{\mrm{AT}} $ implies that the assignment $x \mapsto \langle x,x \rangle_{\mrm{AT}}$ is a \emph{homomorphism}. Tautologically, $\langle  \cdot, \cdot \rangle_{\mrm{AT}}$ is alternating if and only if this homomorphism is $0$. The strategy is to rewrite this homomorphism in terms of canonical cohomology operations called the \emph{Steenrod squares}. Motivated by classical results on the algebraic topology of manifolds, we then develop a theory of ``Stiefel-Whitney classes'' in \'{e}tale cohomology of algebraic varieties which facilitates the calculation of the relevant Steenrod squares. 	

\begin{remark}\label{rem: analogy}
Our approach is guided by an analogy between the Artin-Tate pairing and the \emph{linking form} on a closed manifold of odd dimension. (See \S \ref{subsec: analogy linking form} for an explanation of this analogy.) Our method applies equally well to the topological situation, and it gives a necessary and sufficient criterion for the linking form to be alternating (see  \S \ref{subsec: alternating linking form}), which to our knowledge does not already exist in the topology literature. 
\end{remark}

%\subsection{Further directions} We speculate on a few applications and extensions of the ideas in this paper. 

%\subsubsection{Potential applications} The fact that the Cassels-Tate pairing is alternating for \emph{elliptic curves} is used in \cite{NekPl}, \cite{NekIII} to establish partial progress on the Parity Conjecture (a parity version of the BSD Conjecture) for elliptic curves.\footnote{We thank Chris Skinner for telling us about these applications.} It therefore seems reasonable to hope that Theorem \ref{thm: proof of conj} may open up analogous attacks on a parity version of the Artin-Tate Conjecture, although such applications are not pursued in the present paper.

%\subsubsection{The case $p=2$}
%Although the original Artin-Tate pairing was only defined on elements of torsion order prime to the characteristic, Milne used flat cohomology to define a non-degenerate skew-symmetric pairing on $\Br(X)_{\nd}[p^{\infty}]$ in \cite{Milne75}. It remains open whether or not this pairing is alternating for $p=2$, but based on the results of this paper and \cite{LLR05} we conjecture that it is. It seems likely that the methods of this paper, adapted to a suitable $p$-adic cohomology theory, will be useful for resolving this conjecture. 

\subsection{Overview of the proof}\label{subsec: overview}

We now give a more detailed outline of our proof of Theorem \ref{thm: main}.

\subsubsection*{Step 1: Reduction to an auxiliary pairing}

 In \S \ref{sec: pairings} we explain that there is a surjection 
\[
H^{2d}_{\et}(X; \Z/2^n\Z(d)) \twoheadrightarrow \Br(X)_{\nd}[2^n],
\]
so it suffices to prove that the pulled-back pairing on $H^2_{\et}(X; \Z/2^n\Z(d))$, which we denote $\langle \cdot, \cdot \rangle_n$, is alternating for all $n$. (Here $\Z/2^n\Z(d)$ is the constant sheaf $\Z/2^n \Z$ with an order $d$ Tate twist.) One reason the pairing $\langle \cdot, \cdot \rangle_n$ is more tractable to study is that the coefficients $\Z/2^n \Z$ carry a \emph{ring structure}, unlike $\Q_2/\Z_2$. As a consequence, the cohomology groups are enhanced with the structure of cohomology operations, which we exploit in the next step. 

\subsubsection*{Step 2: Expression in terms of cohomology operations}

We would like to understand the canonical linear functional $x \mapsto \langle x, x \rangle_n$ on $H^{2d}_{\et}(X; \Z/2^n\Z(d)) $. The key observation is that it can be expressed in terms of certain cohomology operations. To convey the spirit of this, we illustrate the flavor of the cohomology operations involved. 

One is the Bockstein operation $\beta$, which is the boundary map 
\[
 \beta \co H^{2d}_{\et}(X; \Z/2^n\Z(d)) \rightarrow H^{2d+1}_{\et}(X; \Z/2^n\Z(d))
 \]
 induced by the short exact sequence of sheaves
\[
0 \rightarrow \Z/2^n\Z(d) \rightarrow \Z/2^{2n}\Z(d) \rightarrow \Z/2^{n}\Z(d) \rightarrow 0.
\]

The second operation is a Steenrod operation, which could thought of informally as a kind of derived enhancement of the squaring operation. Concretely, the operation in question can be described in the following way (but see \S \ref{sec: steenrod} for a formal definition). Let $C_{\et}^*(X)$ be ``the'' \'{e}tale cochain complex computing $H^*_{\et}(X)$, which has the cup product
\[
 C^*_{\et}(X) \otimes C^*_{\et}(X) \xrightarrow{\smile} C^*_{\et}(X)
\]
Since the cup product on $H^*_{\et}(X)$ is graded commutative, we can find a chain homotopy 
\[
\mrm{cup}_1 \co C^*_{\et}(X) \otimes C^*_{\et}(X) \rightarrow C^{*-1}_{\et}(X)
\]
such that 
\[
d\mrm{cup}_1(u \otimes v)+\mrm{cup}_1 (d(u \otimes v)) = u \smile v \mp  v \smile u. 
\]
If $u \in Z^{2d+1}_{\et}(X)$, then $2^{n-1}\mrm{cup}_1(u \otimes u) $ defines a cohomology class in $H^{4d+1}_{\et}(X)$, if the coefficients are $2^n$-torsion. For $n=1$, the map $[u] \mapsto [2^{n-1} \mrm{cup}_1(u \otimes u)]$ is the Steenrod square $\Sq^{2d}$. For larger $n$, it is a cohomology operation that we call $\wt{\Sq}^{2d}$. These generalized Steenrod squares are carefully defined and studied in \S \ref{sec: steenrod}.

The key identity that we have referred to is the following (the precise statement is Theorem \ref{thm: pairing identity}): 

\begin{thm}
For all $x\in H^{2d}_{\et}(X; \Z/2^n\Z(d))$ we have 
\[
\langle x, x\rangle_n = \wt{\Sq}^{2d} (\beta (x)).
\] 
\end{thm}

 The theorem is proved in \S \ref{sec: bockstein} using cohomology operations that we call ``higher Bockstein operations'', which are those arising in the ``Bockstein spectral sequence'' for $\Z/2^n \Z(2d)$. The argument is a little elaborate: we directly calculate the difference between the two sides as a differential in the spectral sequence. The game is then to deduce indirectly that this differential must vanish, by using Poincar\'{e} duality to infer information about the behavior of the $E_{\infty}$-page.

\subsubsection*{Step 3: Relation to characteristic classes} 
The previous step reduces us the problem to that of understanding $\wt{\Sq}^{2d}$ sufficiently well. A consequence of the structure of the $\wt{\Sq}^{2d}$ discussed in \S \ref{sec: steenrod} is that we only need to calculate the effect of the classical Steenrod square $\Sq^{2d}$, which operates on cohomology with $\Z/2\Z$-coefficients. For this purpose we draw inspiration from a theorem for smooth manifolds due originally to Wu (the precise version is explained in \S \ref{sec: Wu}):

\begin{thm}[Wu]
Let $M$ be a closed smooth manifold of dimension $d$. For $x \in H^{d-i}(M; \Z/2\Z)$, we have
\[
\Sq^i x = P(w_1, w_2, \ldots) \smile x
\]
where $P$ is some explicit polynomial and the $w_j$ are the Stiefel-Whitney classes of $TM$.
\end{thm}

Our goal in this step is to establish a version of Wu's theorem for the \'{e}tale cohomology of smooth projective varieties over $\F_q$. The first task is to define an appropriate notion of Stiefel-Whitney classes, which is the subject of \S \ref{sec: SW}. Next, we establish an \'{e}tale-cohomological analogue of Wu's theorem in \S \ref{sec: Wu}. The overarching meta-strategy of the proofs is to attempt to imitate the theory as developed in algebraic topology. However there are a few possibly surprising subtleties, which result in this being the most technical part of the paper. For example, our argument employs the apparatus of \emph{relative \'{e}tale homotopy theory} developed by Harpaz-Schlank \cite{HS} and Barnea-Schlank \cite{BS}, following in the tradition of Artin-Mazur and Friedlander. Hence our Theorem \ref{thm: proof of conj} is, in our humble opinion, a rather compelling example of how this abstract theory can be used to understand very concrete questions which have no apparent grounding in homotopy theory.

\subsubsection*{Step 4: Calculation of characteristic classes}

The upshot of the preceding steps is that we can express the obstruction for $\langle  \cdot, \cdot \rangle_{\mrm{AT}}  $ to be alternating explicitly in terms of our ``\'{e}tale Stiefel-Whitney classes''. We then need to show that this obstruction actually vanishes. After some elementary manipulations, it becomes clear that the key issue is whether or not a certain explicit polynomial in Stiefel-Whitney classes, which is a cohomology class with coefficients mod $2$, lifts to an integral class. This calculation is carried out in \S \ref{sec: alternating}. Motivated by an analogous fact for complex manifolds, we prove a formula expressing our Stiefel-Whitney classes in terms of Chern classes, and conclude that they lift because Chern classes do, which establishes Theorem \ref{thm: main}.

%\subsection{A speculation}
%In characteristic $p$ the group $\Br(X)[p^{\infty}]$ is not addressed by any of our results (we do not even have a pairing defined on it, yet). However, the argument of \cite{LLR05} (with the correction by \cite{Geiss}) implies that if the Birch and Swinnerton-Dyer Conjecture is true, then $\# \Br(X)[p^{\infty}]$ is also a perfect square. One could ask for the unconditional theorem that $\#\Br(X)_{\nd}[p^{\infty}]$ is a perfect square. In view of the story here, it seems natural to speculate that there exists a non-degenerate alternating pairing on $\Br(X)_{\nd}[p^{\infty}]$ for $X$ in characteristic $p$ that ``explains'' this numerology. 

\subsection{Comparison with earlier arguments}
The idea to use Steenrod squares on this problem goes back to Zarhin, who in \cite{Zar} studied the case when the surface admits a lift to characteristic $0$. Zarhin's argument was pushed further by Urabe in \cite{Urabe96} to show that $\# \Br(X)_{\nd}[2^{\infty}]$ is a perfect square. 

Our argument also uses Steenrod operations, although both the operations used and the manner of use are quite different. The Steenrod operation at the focus of \cite{Zar} and \cite{Urabe96} is actually just the squaring operation on an element of $H^2_{\et}(X_{\ol{\F}_q}; \Z/2\Z)$. By contrast, our argument studies a deeper connection between Tate's pairing with subtler cohomology operations, as articulated in Theorem \ref{thm: pairing identity}. 

Our strategy to ``compute'' these operations is based on setting up an analogue of Wu's theorem, as in \cite{Zar} and \cite{Urabe96}. The content of \S \ref{sec: SW}, especially the idea to lift Stiefel-Whitney classes to Chern classes, was inspired by \cite{Urabe96}. However, the proof of such an analogue is significantly more difficult, because we work with arithmetic (as opposed to geometric) \'{e}tale cohomology. For example our arithmetic version of Wu's Theorem is no easier to prove in the case when $X$ lifts to characteristic $0$, whereas Urabe's geometric version follows immediately from the Artin Comparison theorem and the classical Wu Theorem in topology.

\subsection{Acknowledgements} This project was conceived after hearing a comment of Akshay Venkatesh on the analogy between the Artin-Tate pairing and the linking form on a 5-manifold. I thank Akshay for his inspirational remark, and also for subsequent discussions on this work. 

It is a pleasure to acknowledge Soren Galatius for teaching me much of the topology which is employed here, for pointing me to several key references, and for answering my questions patiently and thoroughly. 

I also thank Levent Alpoge, Aravind Asok, Shachar Carmeli, Tom Church, Christopher Deninger, Marc Hoyois, Arpon Raksit, Arnav Tripathy, and Kirsten Wickelgren for conversations related to this paper. This document benefited enormously from comments, corrections, and suggestions by Soren Galatius, Akshay Venkatesh, and an anonymous referee. Finally, significant parts of this research were carried out while I was a guest at the Institute for Advanced Study in 2017, and supported by an NSF Graduate Fellowship.

\section{Pairings for varieties over finite fields}\label{sec: pairings}

\subsection{The Artin-Tate pairing}\label{subsec: Tate pairing}

 We briefly summarize the definition of the generalized pairing $\langle  \cdot, \cdot \rangle_{\mrm{AT}} $ from \cite[\S 2 and \S 3]{Jahn15}. Let $X$ be a geometrically connected, smooth, projective variety of even dimension $2d$ over $\F_q$. Jahn defines the \emph{higher Brauer group}
 \[
 \Br^d(X):= H_L^{2d+1}(X; \Z(d))
 \]
 where $H_L$ denotes Lichtenbaum cohomology. By \cite[Lemma 1]{Jahn15} we have the following interpretation of its non-divisible quotient for $\ell \neq p$:
 \[
 \Br^d(X)_{\nd}[\ell^{\infty}] \cong H_{\et}^{2d+1}(X; \Z_{\ell}(d))_{\tors}. 
 \]

 The pairing $\langle  \cdot, \cdot \rangle_{\mrm{AT}} $ on $ \Br^d(X)_{\nd}[\ell^{\infty}] $ is defined as follows. For any abelian group $G$, let $G_{\nd}$ denote its non-divisible quotient (i.e. the quotient by the maximal divisible subgroup). Let 
\[
\wt{\delta} \co H_{\et}^{2d}(X; \Q_{\ell}/\Z_{\ell}(d))_{\nd} \rightarrow H_{\et}^{2d+1}(X; \Z_{\ell}(d))_{\tors}
\]
be the boundary map induced by the short exact sequence 
\[
0 \rightarrow \Z_{\ell}(d) \rightarrow \Q_{\ell}(d)\rightarrow  \Q_{\ell}/\Z_{\ell}(d) \rightarrow 0.
\]
The map $\wt{\delta}$ is an isomorphism, so it suffices to define a pairing on $H_{\et}^{2d}(X; \Q_{\ell}/\Z_{\ell}(d))_{\nd}$. Now the key point is that $X$ has a Poincar\'{e} duality of dimension $4d+1$, since $X_{\ol{\F}_q}$ has a Poincar\'{e} duality of dimension $4d$ and  $\Spec\F_q$ has a Poincar\'{e} duality of dimension $1$. (This may be deduced directly from the Hochschild-Serre spectral sequence and the usual Poincar\'{e} duality for $X_{\ol{\F}_q}$.) In particular, there is a canonical isomorphism 
\[
\int \co H^{4d+1}_{\et}(X; \Q_{\ell}/\Z_{\ell}(2d))  \xrightarrow{\sim} \Q_{\ell}/\Z_{\ell}.
\]

\begin{defn}
For $x,y \in H^{2d}_{\et}(X; \Q_{\ell}/\Z_{\ell}(d))_{\nd}$, we define 
\[
\langle x, y \rangle_{\mrm{AT}} := \int (x \smile \wt{\delta} y).
\]
\end{defn}

From Poincar\'{e} duality and the fact that $\wt{\delta}$ is an isomorphism, it is evident that this pairing is non-degenerate. It is also skew-symmetric - this is proved in \cite[\S 3]{Jahn15}, and it also follows from combining Proposition \ref{prop: skew-symmetric} and Proposition \ref{prop: compatibility} below. 

\subsection{The analogy to the linking form}\label{subsec: analogy linking form}

An analogous pairing exists on any orientable manifold $M$ of odd dimension $4d+1$, and is called the linking form. Our approach was inspired by ideas of Browder used to study a variant of the linking form in \cite{Browdb}.

 Actually, to make the analogy sharper it is better to work in a slightly more general setup. We do not assume that $M$ is orientable, but we do assume that the orientation sheaf of $M$ is given by the tensor square of a $\Z_{\ell}$-local system $\Cal{L}$. Then there is a pairing on $H^{2d}(M; \Q_{\ell}/\Z_{\ell} \otimes \Cal{L})_{\nd} \cong H^{2d+1}(M; \Z_{\ell} \otimes \Cal{L})_{\tors}$ given by 
\[
\langle x, y \rangle := \int x \smile \wt{\delta} y,
\]
where $\wt{\delta}$ is the analogous boundary map to that in \S \ref{subsec: Tate pairing} and
\[
\int \co H^{4d+1}(M; \Q_{\ell}/\Z_{\ell} \otimes \Cal{L}^{\otimes 2})_{\nd} \xrightarrow{\sim} \Q_{\ell}/\Z_{\ell}
\]
is the isomorphism furnished by Poincar\'{e} duality. The same argument shows that this pairing is skew-symmetric (this is why we require the dimension to be $1 \mod{4}$). This pairing is called the \emph{linking form}\footnote{Strictly speaking, this is the $\ell$-primary part of the usual linking form.}. 

We were informed by an anonymous referee that the linking form on an orientable smooth 5-manifold should be alternating if and only if the manifold admits a spin$^{c}$-structure. We were not previously aware of this fact, nor have we have been able to locate a proof in the literature, but after hearing it we realized that our method yields a necessary and sufficient criterion for the linking form on any odd-dimensional topological manifold (with orientation sheaf of the above form) to be alternating, which recovers the aforementioned result for orientable smooth 5-manifolds. This will be explained in \S \ref{subsec: alternating linking form}. Although our paper is phrased for \'{e}tale cohomology, the reader can check that every one of the results has a corresponding statement for the singular cohomology of manifolds, which is either easier to prove or already a known theorem.

\subsection{An auxiliary pairing}

We define an auxiliary pairing on the group $H_{\et}^{2d}(X; \Z/2^n\Z(d))$. As in \S \ref{subsec: Tate pairing} there is a Poincar\'{e} duality for $H_{\et}^*(X; \Z/2^n\Z(*))$, which means in particular that there is a fundamental class inducing an isomorphism
\[
\int \co H_{\et}^{4d+1}(X; \Z/2^n\Z(2d)) \xrightarrow{\sim} \Z/2^n\Z.
\]

\begin{defn}\label{def: pairing mod 2^n} We have the short exact sequence of sheaves on $X$:
\[
0 \rightarrow \Z/2^n\Z(d) \rightarrow  \Z/2^{2n} \Z(d) \rightarrow \Z/2^n\Z(d) \rightarrow 0
\]
inducing a boundary map 
\[
\beta \colon H^i_{\et}(X; \Z/2^n\Z(d)) \rightarrow H^{i+1}_{\et}(X; \Z/2^n\Z(d)).
\] 
We define the pairing 
\[
\langle \cdot , \cdot \rangle_{n} \colon H^{2d}_{\et}(X; \Z/2^n\Z(d)) \times H^{2d}_{\et}(X; \Z/2^n\Z(d)) \rightarrow \Z/2^n\Z
\]
by 
\[
\langle x, y \rangle_n := \int x \smile \beta y .
\]
\end{defn}

\begin{prop}\label{prop: skew-symmetric}
The pairing $\langle \cdot , \cdot \rangle_n$ is skew-symmetric. 
\end{prop} 

\begin{proof}
The assertion is equivalent to 
\[
x \smile \beta y  + y \smile \beta  x  = 0.
\]
Since $\beta$ is a derivation, we have $ x \smile \beta y  + y \smile \beta  x  = \beta(x \smile y)$. Then the result follows from the next Lemma. 
\end{proof}

\begin{lemma}\label{lem: top boundary vanishes}
The boundary map $\beta  \colon H^{4d}_{\et}(X; \Z/2^n\Z(2d)) \rightarrow  H^{4d+1}_{\et}(X; \Z/2^n\Z(2d))$ is $0$. 
\end{lemma}

\begin{proof}
By the obvious long exact sequence, the image is the kernel of 
\[
[2^{n}] \co H^{4d+1}_{\et}(X; \Z/2^n\Z(2d)) \rightarrow H^{4d+1}_{\et}(X; \Z/2^{2n}\Z(2d))
\]
which is identified with the inclusion $2^n \Z/2^{2n}\Z \hookrightarrow \Z/2^{2n}\Z$ by Poincar\'{e} duality. 
\end{proof}

\begin{prop}\label{prop: compatibility}
The boundary map $H^{2d}_{\et}(X; \Z/2^n\Z(d)) \rightarrow H^{2d+1}_{\et}(X; \Z_2(d))$ induced by the short exact sequence 
\[
0 \rightarrow \Z_2(d)  \xrightarrow{2^n} \Z_2(d) \rightarrow \Z/2^n\Z(d) \rightarrow 0
\]
surjects onto $H^{2d+1}_{\et}(X, \Z_2(d))[2^n]$. Moreover, it is compatible for the pairings $\langle \cdot , \cdot \rangle_n$ and $\langle \cdot , \cdot \rangle_{\mrm{AT}}$ in the sense that the following diagram commutes
\[
\xymatrix @C=0pc{
H^{2d}_{\et}(X; \Z/2^n\Z(d)) \ar@{->>}[d] & \times &  H^{2d}_{\et}(X; \Z/2^n\Z(d))  \ar@{->>}[d]  \ar[rrrrrrr]^{\langle \cdot , \cdot \rangle_n} &&&&&&& H^{4d+1}_{\et}(X; \Z/2^n\Z(d)) \ar[d]_{\sim}  \\
H^{2d+1}_{\et}(X, \Z_2(d))[2^n] & \times &H^{2d+1}_{\et}(X; \Z_2(d))[2^n]  \ar[rrrrrrr]_{\langle \cdot , \cdot \rangle_{\mrm{AT}}} &&&&&&& H^{4d+1}_{\et}(X; \Q_2/\Z_2(d)) [2^n]
	}
\]
\end{prop}

\begin{proof} The first claim is immediate from the long exact sequence. For the second claim, we will apply the following observation, which is an immediate consequence of naturality for the cup product: given a map of short exact sequences of sheaves
\[
\begin{tikzcd}
0 \ar[r] & A \ar[r] \ar[d, "f"] & B \ar[r] \ar[d, "g"] & C \ar[r] \ar[d, "h"] & 0 \\
0 \ar[r] & A' \ar[r] & B' \ar[r] & C' \ar[r] & 0
\end{tikzcd}
\]
and multiplications fitting into a commutative diagram 
\[
\begin{tikzcd}
A \otimes  C \ar[r] \ar[d, "f \otimes h" ]& C \ar[d, "h"]\\
A' \otimes C' \ar[r] & C'
\end{tikzcd}
\]
then for $a \in H^*(X;A), c \in  H^*(X;C)$, we have $h(a \smile c) = f(a) \smile h(c)$. 

We apply this observation to each of the maps of short exact sequences in the following commutative diagram of sheaves:
\[
\begin{tikzcd}
0 \ar[r] & \Z/2^n\Z(d) \ar[r, "2^n"] & \Z/2^{2n} \Z(d) \ar[r] & \Z/2^n\Z(d) \ar[r] & 0 \\
0 \ar[r] & \Z_2 (d)\ar[u] \ar[r, "2^n"] \ar[d, equals] & \Z_2 (d)\ar[u]  \ar[r] \ar[d] & \Z/2^n\Z(d)  \ar[u, equals] \ar[r] \ar[d, "\frac{1}{2^n}"]& 0 \\
0 \ar[r] & \Z_2(d)\ar[r] & \Q_2 (d) \ar[r] & \Q_2/\Z_2 (d)\ar[r] & 0
\end{tikzcd}
\]
Denote by $\wt{\beta}$ the boundary map in cohomology corresponding to the middle horizontal sequence, recalling that $\beta$ and $\wt{\delta}$ denote the boundary maps for the top and bottom horizontal sequences, respectively. The observation applied to the upper map of sequences shows that for $x, y \in H^*_{\et}(X; \Z/2^n\Z(d))$ we have 
\[
x \smile \beta (y) = x \smile \wt{\beta}(y).
\]
The observation applied to the lower map of sequences shows that 
\[
x \smile  \wt{\beta}(y) \mapsto [\frac{1}{2^n}]_*(x) \smile \wt{\delta}(y). 
\]
Combining these equations yields the desired conclusion. 
\end{proof}

Proposition \ref{prop: compatibility} immediately implies: 

\begin{cor}\label{cor: alternating_criterion}
If the pairing $\langle \cdot , \cdot \rangle_n$ on $H_{\et}^{2d}(X; \Z/2^n\Z(d))$ is alternating then so is the pairing $\langle \cdot, \cdot \rangle_{\mrm{AT}}$ on $\Br^d(X)_{\nd}[2^n]$. 
\end{cor}

%Since $\Br^d(X)_{\nd} \cong H^{2d+1}_{\et}(X; \Z_2(d))_{\tors}$ is finite, we have $\Br^d(X)_{\nd}[2^{\infty}] = \Br^d(X)_{\nd}[2^n]$ for some finite $n$. (However, this observation is unnecessary for our purposes.) 
Hence to prove Theorem \ref{thm: main} we are reduced to proving: 

\begin{thm}\label{thm: pseudo_main}
The pairing $\langle \cdot , \cdot \rangle_n$  is alternating for all $n$. 
\end{thm}

The proof of Theorem \ref{thm: pseudo_main} will be the focus of the rest of the paper.

\section{Steenrod squares}\label{sec: steenrod}

In this section we define the (generalized) Steenrod squares in  \'{e}tale cohomology and establish the key facts about them. The perspective we adopt here is that we can define our cohomology operations on topological spaces, and then transport them to \'{e}tale cohomology via \'{e}tale homotopy theory. 

Let us emphasize that our construction is certainly not original to this paper. (We do introduce some generalized operations $\wt{\Sq}^i$ that we have not seen defined elsewhere, but they are minor variants of the Steenrod squares.) The earliest construction of Steenrod squares which was general enough to apply to \'{e}tale cohomology occurs in work of Epstein \cite{Eps66}. Our definition is perhaps closer to (a special case of) Jardine's construction in \cite{Jar89}.

\subsection{The Steenrod algebra in topology}

We begin with a motivational pitch about Steenrod operations in algebraic topology. An old and fundamental observation in algebraic topology is that the singular cohomology of any space with $\Z/2\Z$ coefficients carries a natural module structure over a ring called the \emph{Steenrod algebra}, and that it is fruitful to understand this additional structure. The Steenrod algebra may be characterized abstractly as the algebra of stable cohomology operations on $H^*(-; \Z/2\Z)$, i.e. all natural transformations $H^j(-; \Z/2\Z) \rightarrow H^k(-; \Z/2\Z)$ commuting with the suspension isomorphisms. 

More concretely, one can exhibit a set of cohomology operations $\Sq^i$ which generate the Steenrod algebra and which admit an explicit description in terms of homotopies defined on the cochain complex of a topological space, whose existence has to do with the failure of the cup product to be commutative at the level of cochains. This will be explained in \S \ref{subsec: cup-i}.

A key point in this paper is that we can and should ask about the analogous structure for $H^*(-; \Z/2^n\Z)$ for every $n$. In particular, we need analogues of the $\Sq^i$ for $\Z/2^n \Z$-coefficients. This leads to a construction of operations that we call $\wt{\Sq}^i$. These turn out to all be induced by the $\Sq^i$, so they are not fundamentally new operations. However, they do come up very directly in our calculations, so it will be useful to spell them out explicitly.

\subsection{\'{E}tale homotopy theory}\label{subsec: etale homotopy}

Using \'{e}tale homotopy theory, we will be able to transport our definition of (singular) cohomology operations on topological spaces to \'{e}tale cohomology of algebraic varieties. Here we just summarize the facts that we need.

To any locally noetherian scheme $X$ there is attached a pro-object in simplicial sets which is called its \emph{\'{e}tale topological type}\footnote{For the construction of $\Et(X)$ one also makes a choice of sufficiently many separably closed fields so that every residue field of a point of $X$ is contained in one of them, but we can ignore this technicality.}, and which we denote $\Et(X)$.  We refer to \cite[Definition 4.4]{Fr82} for the definition of $\Et(X)$. Given the awkwardness of the expression ``pro-(simplicial set)'', we will henceforth use the phrase ``pro-space'' to denote a pro-object in simplicial sets (however, it will be important at certain points that our ``spaces'' are really simplicial sets).

\begin{defn}We define the category of \emph{local coefficient systems} on a pro-space $\{T^i \co i \in I\}$ as follows. 
\begin{itemize}
\item An object is a local coefficient system on some $T^j$.
\item A map between local coefficient systems, defined by $\Cal{L}_1$ on $T^i$ and $\Cal{L}_2$ on $T^j$, is a map between the pullbacks of $\Cal{L}_1$ and $\Cal{L}_2$ to $T^k$ for some $k>i,j$. 
\end{itemize}
\end{defn}

\begin{remark}
In Friedlander's original definition \cite[\S 5, p. 48]{Fr82}, a ``local coefficient system'' is an isomorphism class of objects in our definition.
\end{remark}

\begin{prop}[{\cite[Corollary 5.8]{Fr82}}]\label{prop: bijection local systems} There is an equivalence of categories between locally constant sheaves on the \'{e}tale site of $X$, and local coefficient systems on the pro-space $\Et(X)$. 
\end{prop}

\begin{defn} 
We define the \emph{cochain complex} of a pro-space $\{T^i\}$ with coefficients in a local coefficient system $\Cal{F}$ to be the direct limit of the levelwise cochain complexes:
\[
C^*(\{T^i\}; \Cal{F}) := \varinjlim_i C^*(T^i; \Cal{F}).
\]
By the exactness of filtered colimits, we have 
\[
H^*(C^*(\{T^i\}; \Cal{F})) \cong \varinjlim_i H^*(C^*(\{T^i\}; \Cal{F})),
\]
so this recovers the definition of the \emph{cohomology} of a pro-space $\{T^i\}$ in \cite[Definition 5.1]{Fr82} as the direct limit of the levelwise cohomology.

 In particular, if $\Et(X)= \{U^i \co i \in I\}$ then
 \[
 C^*(\Et(X); \Cal{F}) := \varinjlim_i C^*(U^i; \Cal{F}) \quad
\text{and}  \quad  H^*(\Et(X); \Cal{F}) := \varinjlim_i H^*(U^i; \Cal{F}).
\]
\end{defn}

\begin{prop}[{\cite[Proposition 5.9]{Fr82}}] If $\Cal{F}$ is a locally constant sheaf on $X$ and $\Et(\Cal{F})$ is the corresponding local coefficient system on $\Et(X)$ under the equivalence of categories in Proposition \ref{prop: bijection local systems}, then there is a natural isomorphism
\[
H^*_{\et}(X; \Cal{F}) \cong H^*(\Et(X); \Et(\Cal{F})).
\]
\end{prop}

\subsection{Steenrod's cup-$i$ product}\label{subsec: cup-i}
 Let $X$ be a topological space. Let $R$ be a local coefficient system in commutative rings, and $C^*(X;R)$ the singular cochain complex. Steenrod defined sequence of maps 
\begin{align*}
\mrm{cup}_i \co C^r(X;R) \otimes C^s(X;R) &\rightarrow C^{r+s-i}(X;R) \\
u \otimes v & \mapsto u \smile_i v 
\end{align*}
called the ``cup-$i$ products''. We will give a high-level exposition; a reference for this standard (in topology) material is \cite[Chapter 2]{MoTan68}.

The cup product for $X$ is induced at the level of chain complexes by the composition of the Alexander-Whitney map and the restriction to the diagonal: 
\begin{equation}\label{eq: rest mult}
C^*(X;R) \leftarrow C^*(X \times X; R) \xleftarrow{\sim} C^*(X;R) \otimes C^*(X;R).
\end{equation}
This composition is not $S_2$-equivariant because the Alexander-Whitney map is not $S_2$-equivariant; it is only $S_2$-equivariant up to homotopy. However, there is a way to rectify it to be an $S_2$-equivariant quasi-isomorphism, which we now describe. 

Let $\mrm{E}S_2$ be a contractible space with a free $S_2$-action; in fact, let us take the explicit model $\mrm{E}S_2 = S^{\infty}$. We view $C_*(\mrm{E}S_2; \ul{R})$ as a cochain complex in non-positive degrees, which provides a free resolution as $S_2$-modules of the constant sheaf $\ul{R}$ in degree $0$. Then there \emph{is} an $S_2$-equivariant quasi-isomorphism
\begin{equation}\label{eq: equivariant chain map}
C^*(X \times X; R) \simeq C^*(X;R) \otimes C^*(X;R) \otimes C_*(\mrm{E}S_2; \ul{R})
\end{equation}
where the $S_2$ action on the right side is diagonal for the ``swap'' action on $C^*(X;R) \otimes C^*(X;R)$ and the tautological action on $\mrm{E}S_2$.  Tensoring \eqref{eq: equivariant chain map} with $C^*(\mrm{E}S_2; \ul{R})$ and applying the evaluation pairing yields an $S_2$-equivariant cochain map
\begin{equation}\label{eq: equivariant pullback map}
C^*(X; R) \otimes  C^*(\mrm{E}S_2; \ul{R}) \leftarrow C^*(X; R) \otimes C^*(X; R)
\end{equation}
where the $S_2$-action is via ``swap'' on the right hand side, and the tautological action on $C^*(\mrm{E}S_2; \ul{R})$ on the left hand side.

%\begin{remark}
%In \cite{MoTan68} the map \eqref{eq: equivariant pullback map} is produced using an equivariant version of the acyclic carrier theorem. The discussion there is stated for $R= \Z$ (constant coefficient system), but obviously goes through unchanged for any local coefficient system of rings (the reader will readily check that in the proof of \cite{MoTan68} Chapter 2 Theorem 1 and ensuing discussion, no reference is made to the coefficients).
%\end{remark}

Now we use the presentation of $S^{\infty}$ as a simplicial complex with two cells $d_i$ and $Td_i$ in every dimension which are interchanged under the $S_2$-action. In the chain complex $C_*(\mrm{E}S_2; \ul{R})$ we then have two corresponding generators $e_i \otimes 1$ and $Te_i \otimes 1 \in C_i(\mrm{E}S_2; \ul{R})$. Contracting \eqref{eq: equivariant pullback map} with $e_i \otimes 1$ gives the cup-$i$ product
\[
C^{r+s-i} (X;R) \leftarrow C^r(X; R) \otimes C^s(X; R) \co \mrm{cup}_i.
\] 
We will also use the notation 
\[
u  \smile_i v := \mrm{cup}_i(u \otimes v).
\]

 We have the coboundary formula \cite[Chapter 2, p. 16]{MoTan68}
\begin{equation}\label{eq: coboundary formula}
d (u  \smile_i v) = (-1)^i d  u \smile_i v + (-1)^{i+r} u  \smile_i d v - (-1)^i u \smile_{i-1} v - (-1)^{rs} v \smile_{i-1} u,  
\end{equation}
where $|u|=r, |v|=s$. We can rewrite \eqref{eq: coboundary formula} as:
\begin{equation}\label{eq: homotopy formula}
d(u \smile_i v) - (-1)^i  \mrm{cup}_i(d(u \otimes v)) = (-1)^{i-1} u \smile_{i-1} v - (-1)^{rs} v \smile_{i-1}u. 
\end{equation}
It is the case $i=1$ in \eqref{eq: homotopy formula} that will be most important for us. For concreteness, let us spell out the informal meaning of \eqref{eq: homotopy formula}. The cup-$0$ product is just the multiplication on cochains. The cup-$1$ product furnishes a chain homotopy between $u \smile_0 v$ and $\pm v \smile_0 u$ ``witnessing'' the graded commutativity of the cup product. The cup-$2$ product furnishes a chain homotopy between $u \smile_1 v$ and $\pm v \smile_1 u$, etc. 

\begin{remark}
In order to bootstrap the $\mrm{cup}_i$-product from simplicial sets to \'{e}tale topological type as in \S \ref{subsec: etale homotopy}, we take a model for the $\mrm{cup}_i$-product which is \emph{functorial} in maps of simplicial sets, whose existence is guaranteed by \cite[Appendix B]{Sm15}. (Although the $\mrm{cup}_i$-product in \cite{Sm15} is phrased with integral coefficients, it exists for any local coefficient system of commutative rings, by base change.) 
\end{remark}

 We now turn to the task of extracting cohomology operations out of the cup-$i$ product. The cup-$i$ product does not preserve cocycles, except in characteristic $2$, so that is the simplest case in which we get cohomology operations, and we discuss it first.

\subsection{Classical Steenrod squares}\label{subsec: classical Steenrod}

If $2=0$ in $R$, then it is easily checked from \eqref{eq: coboundary formula} that the operation 
\[
u \mapsto u \smile_i u
\]
sends cocycles to cocycles and coboundaries to coboundaries, hence descends to a cohomology operation 
\[
\Sq_i \co H^r(X; R) \rightarrow H^{2r-i}(X; R). 
\]
We then define the Steenrod square 
\[
\Sq^i := \Sq_{r-i} \co H^r(X; R) \rightarrow H^{r+i}(X; R).
\]
 For $R=\Z/2\Z$, which is the case studied in \cite{MoTan68} \S 2, this construction recovers the classical Steenrod squares.
 
 \subsubsection*{Properties of the Steenrod squares} We now recall the formal properties of these classical Steenrod squares. (Proofs can be found in \cite[\S 2,3]{MoTan68}.)

\begin{enumerate}
\item ({\sc naturality}) For any $f \colon X' \rightarrow X$, we have 
\[
f^* \Sq^i = \Sq^i f^*.
\]
\item ({\sc cartan formula}) We have
\[
\Sq^i(x \smile y) = \sum_{j=0}^i \Sq^j(x) \smile \Sq^{i-j}(y).
\]

If we define the \emph{total Steenrod operation} $\Sq := \sum_i \Sq^i$, then the Cartan formula can be neatly packaged as 
\[
\Sq(x \smile y )  = \Sq(x) \smile \Sq(y).
\]

\item ({\sc Adem relations}) If $0<i<2j$ then 
\[
\Sq^i\Sq^j = \sum_{k=0}^{[i/2]} \binom{j-1-k}{i-2k} \Sq^{i+j-k} \Sq^k.
\]
\item ({\sc special cases}) For $x \in H^j(X; \Z/2\Z)$ we have 
\begin{itemize}
\item $\Sq^0(x) = x$, 
\item $\Sq^1(x) = \beta(x)$ for $\beta$ the connecting homomorphism $H^j(X; \Z/2\Z) \xrightarrow{\beta} H^{j+1}(X;\Z/2\Z)$ induced by the short exact sequence $0 \rightarrow \Z/2\Z \rightarrow \Z/4 \Z \rightarrow \Z/2\Z \rightarrow 0$.  
\item $\Sq^j (x) = x \smile x $,
\item For $i>j$, we have $\Sq^i(x)=0$. 

\end{itemize}
\item ({\sc Stability}) The Steenrod operations commute with the suspension isomorphisms 
\[
H^i(X;\Z/2) \cong H^{i+1}(\Sigma X;\Z/2).
\] 
\end{enumerate}

\subsection{Generalized Steenrod squares}\label{subsec: generalized steenrod}

We now drop our assumption that $2=0$ in $R$. If $u$ is a cocycle, we see from \eqref{eq: coboundary formula} that 
\begin{equation}\label{eq: coboundary square}
 d(u \smile_i u) = [(-1)^i  - (-1)^{r^2}] u \smile_{i-1} u . 
\end{equation}

\emph{Suppose $2^n= 0$ in $R$.} If $r-i$ is even, then \eqref{eq: coboundary square} implies that $2^{n-1} u \smile_i u$ is a cocycle. Furthermore, one can  check that the operation $ u\mapsto 2^{n-1} u \smile_i u$ also takes coboundaries to coboundaries, and therefore descends to a cohomology operation
\[
\wt{\Sq}_i \co H^r(X; R) \rightarrow H^{2r-i}(X; R).
\]

\begin{defn}
If $i$ is even, we define 
\[
\wt{\Sq}^i := \wt{\Sq}_{r-i} \co H^r(X; R) \rightarrow H^{r+i}(X; R).
\]
\end{defn}

\begin{lemma}\label{lemma: top generalized sq}
Continue to assume that $2^n = 0 $ in $R$. Let $\red_2 \co H^*(X; R) \rightarrow H^*(X; R/2R)$ be the reduction mod $2$, and let $[2^{n-1}] \co H^*(X; R/2R) \rightarrow H^*(X; R)$ be the map induced by $R/2R \xrightarrow{2^{n-1}} R$. If $i$ is even, then we have
\[
\wt{\Sq}^i = [2^{n-1}] \circ \Sq^i \circ \red_2.
\]
\end{lemma}

\begin{proof}
This is immediate upon unwinding the definitions. 
\end{proof}

Next suppose that $r-i$ is odd. \emph{In this case we do not assume a priori that $2^n=0$ in $R$.} (Although we do not need the operations where $r-i$ is odd in this paper, we construct them for the sake of completeness.) From \eqref{eq: coboundary square} we see that if  $r-i$ is odd, then $u \smile_i u$ is a cocycle if $u$ is a cocycle. Similarly one checks that $u \mapsto u \smile_i u$ sends coboundaries to coboundaries, hence descends to a cohomology operation
\[
\wt{\Sq}_i \co H^r(X; R) \rightarrow H^{2r-i}(X; R).
\]

\begin{defn}
If $i$ is odd, we define 
\[
\wt{\Sq}^i := \wt{\Sq}_{r-i} \co H^r(X; R) \rightarrow H^{r+i}(X; R).
\]
\end{defn}	

Let us elucidate the relationship between the Steenrod squares constructed in the two cases. If $u$ is cocycle, then by \eqref{eq: coboundary formula} we have 
\[
u \smile_i u = (-1)^{i} 2 d(u \smile_{i+1} u).
\]
Then the analogue of Lemma \ref{lemma: top generalized sq} is:

\begin{lemma} Suppose that the complex
\[
0 \rightarrow R/2^nR \xrightarrow{2} R/2^{n+1}R \rightarrow R/2 R \rightarrow 0
\]
is short exact. Let $\beta_{2,2^n} \co H^*(X; R/2R) \rightarrow H^{*+1}(X; R/2^n R)$ be the induced boundary map, and let $\red_2$ be as in Lemma \ref{lemma: top generalized sq}. If $i$ is odd, then we have
\[
\wt{\Sq}^i   =  \beta_{2,2^n} \circ \Sq^{i-1}  \circ \red_2.
\]	
\end{lemma}

\begin{proof}
This is immediate upon unwinding the definitions. 
\end{proof}

 \subsection{Application to \'{e}tale cohomology}
 
 Let $R := \bigoplus_{j \in \Z} \Z/2^n \Z(j)$, viewed as a locally constant sheaf on $X$ (where ``$(j)$'' denotes the Tate twist) valued in rings, with multiplicative structure given by the isomorphisms
 \[
 \Z/2^n \Z(j) \otimes \Z/2^n \Z(j') \xrightarrow{\sim}  \Z/2^n \Z(j+j').
 \]
 Applying \S \ref{subsec: etale homotopy} and the construction of \S \ref{subsec: generalized steenrod}, we obtain operations 
 \[
 \wt{\Sq}^i \co H_{\et}^r(X; \Z/2^n\Z(j)) \rightarrow H_{\et}^{r+i}(X; \Z/2^n \Z(2j)).
 \]
For convenience of the reader, we summarize all the facts that we shall need about the $ \wt{\Sq}^i$ below.

\begin{example}\label{lemma: cup-1}
If $x \in H^{i+1}_{\et}(X; \Z/2^n \Z(j))$ then $\wt{\Sq}^{i}(x)$ has the following description. Let $C^*_{\et}(X;\Z/2^n \Z(j)) $ be the \'{e}tale cochain complex for $X$, defined as in \S \ref{subsec: etale homotopy}. There is a chain homotopy 
\begin{align*}
\mrm{cup}_1 \co   C^r_{\et}(X;\Z/2^n \Z(j)) \otimes C^s_{\et}(X; \Z/2^n \Z(j)) &\rightarrow C^{r+s-1}_{\et}(X; \Z/2^n \Z(2j)) \\
u \otimes v & \mapsto u \smile_1 v 
\end{align*}
such that (by the $i=1$ case of \eqref{eq: homotopy formula}) 
\[
d (\mrm{cup}_1(u \otimes v)) +\mrm{cup}_1( d (u \otimes v)) = u \smile v -  (-1)^{rs} v \smile u.
\]
Let $u \in C^{i+1}_{\et}(X; \Z/2^n \Z(j))$ be a representative for $x$. If $i$ is even, then we have
\[
\wt{\Sq}^{i}(x) = [2^{n-1}u \smile_1 u].
\]
\end{example}

\begin{lemma}\label{lemma: etale generalized sq}
Let $[2^{n-1}] \co H^*_{\et}(X; \Z/2\Z(j)) \rightarrow H^*_{\et}(X; \Z/2^n\Z(j))$ be the map induced by the inclusion of sheaves $\Z/2\Z(j) \xrightarrow{2^{n-1}} \Z/2^{n} \Z(j)$.
Let $\red_2$ be the reduction mod $2$. If $i$ is even, then 
\[
\wt{\Sq}^i = [2^{n-1}] \circ \Sq^i \circ \red_2.
\] 
\end{lemma}

\begin{proof}
This follows from combining Lemma \ref{lemma: top generalized sq}, which implies the formula for all simplicial complexes in particular, and \S \ref{subsec: etale homotopy}, which transports the result to \'{e}tale cohomology.
\end{proof}

\section{Bockstein operations}\label{sec: bockstein}

The goal of this section is to prove Theorem \ref{thm: pairing identity}, which expresses the pairing of Definition \ref{def: pairing mod 2^n} in terms of cohomology operations. Our argument originated from studying a generalized version of the Bockstein spectral sequence, and was motivated by a calculation in \cite{Browd61}. However, we have found it cleaner for expository purposes to present a proof with the language of spectral sequences stripped out. 

\subsection{Higher Bockstein operations}\label{subsec: higher bockstein}

The key technical ingredient in the proof is the study of ``higher Bockstein operations''. These form a family of cohomology operations $\{\beta_r\}$ ``growing off'' of the Bockstein $\beta$ in the following sense. We have
\[
\beta_1 := \beta \co H^*_{\et}(X; \Z/2^n\Z(j)) \rightarrow H^{*+1}_{\et}(X; \Z/2^n\Z(j)).
\]
The operation $\beta_r$ is only defined on the kernel of $\beta_1, \ldots, \beta_{r-1}$, and its image is only defined modulo the image of $\beta_1, \ldots, \beta_{r-1}$. (These operations $\beta_r$ arise as differentials in a spectral sequence, which explains this structure.) In fact we only need $\beta_1$ and $\beta_2$ for our purposes.

\begin{remark}
We must now confront the technical subtlety that $\Z_{\ell}$-\'{e}tale cohomology is not, as defined classically, the cohomology of a cochain complex with $\Z_{\ell}$-coefficients, while the state of affairs is much more naturally reasoned about and phrased in terms of ``integral cochains''. It is straightforward to translate all our statements into ones about compatible systems of $\ell$-adic sheaves; for example the short exact sequence 
\[
0 \rightarrow \Z_\ell(j) \xrightarrow{\ell^n}	 \Z_{\ell}(j)  \rightarrow \Z/\ell^n \Z(j) \rightarrow 0
\]
should be replaced by the system of sequences
\[
0 \rightarrow \Z/\ell^N \Z(j) \xrightarrow{\ell^n}	 \Z/\ell^{N+n}\Z(j)  \rightarrow \Z/\ell^n \Z(j) \rightarrow 0
\]
for all $N \gg 0$. We leave this translation to the reader so as not to complicate our exposition. (It should be also be possible to deal with this problem more uniformly using the pro-\'{e}tale topology of Bhatt-Scholze \cite{BS15}.) 
\end{remark}

Recall that the Bockstein $\beta$ is induced by the short exact sequence of sheaves
\[
0 \rightarrow \Z/\ell^n \Z (j) \xrightarrow{\ell^n}	 \Z/\ell^{2n} \Z(j)  \rightarrow \Z/\ell^n \Z(j) \rightarrow 0.
\]
Thanks to the commutative diagram 
\[
\begin{tikzcd}
C^*(X; \Z_\ell(j)) \ar[d] \ar[r, "\ell^n"] & C^*(X; \Z_\ell(j)) \ar[r] \ar[d] & C^*(X; \Z/\ell^n\Z(j)) \ar[d, equals] \\
C^*(X; \Z/\ell^n\Z(j)) \ar[r] &  C^*(X; \Z/\ell^{2n}\Z(j)) \ar[r] &  C^*(X; \Z/\ell^n\Z(j))
\end{tikzcd}
\] 
it admits the following alternative description. For $x \in H^*_{\et}(X; \Z/\ell^n\Z(j))$ we let $\wt{x} \in C^*_{\et}(X; \Z/\ell^n \Z(j))$ be a representative for $x$, and $a$ a lift of $\wt{x}$ in $C^*_{\et}(X; \Z_{\ell}(j))$. Since $\wt{x}$ is a cocycle, $da$ is divisible by $\ell^n$ within $C^{*+1}_{\et}(X; \Z_{\ell}(j))$, so we may define 
\[
\wt{\beta}(x) := \left[\frac{1}{\ell^n}da \right]\in  H^{*+1}_{\et}(X; \Z_{\ell}(j)).
\]
Then we define 
\[
\beta(x) := \ol{\wt{\beta}(x) }\in H^{*+1}_{\et}(X; \Z/{\ell}^n \Z(j))
\]
to be the reduction of $\wt{\beta}(x)$ mod $\ell^n$. Note that $\wt{\beta}$ is the boundary map for the short exact sequence of sheaves
\[
0 \rightarrow \Z_\ell(j) \xrightarrow{\ell^n}	 \Z_{\ell}(j)  \rightarrow \Z/\ell^n \Z(j) \rightarrow 0.
\]

\begin{defn}\label{defn: beta_2}
We define operations
\begin{align*}
\wt{\beta}_2 \co (\ker \beta_1 \subset H^{*}_{\et}(X; \Z/\ell^n \Z(j)) ) & \rightarrow H^{*+1}_{\et}(X; \Z_{\ell}(j)) / \Ima \wt{\beta}_1 \\
\beta_2  = \ol{\wt{\beta}_2} \co (\ker \beta_1 \subset H^{*}_{\et}(X; \Z/\ell^n \Z(j)) ) &\rightarrow H^{*+1}_{\et}(X; \Z/\ell^n\Z(j)) / \Ima \beta_1
\end{align*}
as follows. If $x \in \ker \beta_1$, then (keeping the notation of the preceding paragraph) we have 
\[
\ol{\frac{1}{\ell^n}da} = d \wt{y} \text{  for some  } \wt{y} \in C^{*}_{\et}(X; \Z/\ell^n \Z(j)),
\]
where the overline means reduction mod $\ell^n$. Hence we may and do choose a lift $b \in C^{*}_{\et}(X; \Z_{\ell} (j)) $ of $\wt{y}$ such that $\frac{1}{\ell^n} d a \equiv d b\mod{\ell^n}$, or in other words 
\[
d a \equiv d (\ell^n b) \mod{\ell^{2n}}.
\]
Then we can form the cochain $\frac{1}{\ell^{2n}} d (a-\ell^n b) \in C^{*}_{\et}(X; \Z_{\ell} (j)) $, which is evidently a cocycle. Finally, we define 
\begin{align*}
\wt{\beta}_2(x) := &\left[\frac{1}{\ell^{2n}} d (a-\ell^n b) \right] \in H^{*+1}_{\et}(X; \Z_{\ell}(j)) / \Ima \wt{\beta}_1, \\
\beta_2(x) := &\ol{\wt{\beta}_2(x) } \in H^{*+1}_{\et}(X; \Z/\ell^n\Z(j)) / \Ima \beta_1.
\end{align*}
We leave it to the reader to check that this is indeed well-defined.
\end{defn}

It is straightforward to define $\beta_r$ in a similar way for all $r$. Since we only need $\beta_1$ and $\beta_2$, we do not spell out the explicit construction. 

%The definition is contained implicitly in the construction of the Bockstein spectral sequence in \S \ref{subsec: BSS}; we view the packaging of exact couples and spectral sequences as the  ``right'' formalism to understand this construction. 	

 In what follows, we will focus on the operations introduced in Definition \ref{defn: beta_2} for $\ell=2$. 

 \begin{prop}\label{prop: beta_2 identity}
Let $\ell=2$ in Definition \ref{defn: beta_2}. For any $x \in H^{2k}_{\et}(X;\Z/2^n\Z(j))$, we have the following identity:
\[
\beta_2(2^{n-1} x^2)= [x \cdot \beta (x) - \wt{\Sq}^{2k} (\beta (x)) ] \in H^{2k+1}_{\et}(X;\Z/2^n\Z(2j))/\Ima \beta.
\]
 \end{prop}

\begin{proof}
Note that since $\beta$ is a derivation, it indeed kills $2^{n-1} x^2$, hence $2^{n-1} x^2$ indeed lives in $\ker \beta$ so that $\beta_2(2^{n-1} x^2)$ is defined.

Let $a$ be any integral cochain in $C^{2k}_{\et}(X; \Z_{2}(j))$ lifting a representative for $x$. Let $y := \beta(x) \in H^{2k+1}_{\et}(X; \Z/2^n \Z(j))$. Then by the definition of $\beta$, we have $da = 2^n b$ where $b\in C^{2k+1}_{\et}(X; \Z_{2}(j))$ lifts a representative for $y$. 

According to the discussion in Definition \ref{defn: beta_2}, $\beta_2(2^{n-1} x^2)$ is calculated by finding an integral cochain lifting a representative for $2^{n-1} x^2$, whose boundary is divisible by $2^{2n}$. We check that $2^{n-1} a^2 + 2^{2n-1} (a \smile_1 b)$ is such an integral cochain, using Example \ref{lemma: cup-1}: 
\begin{align*}
d(2^{n-1} a^2 + 2^{2n-1} (a \smile_1 b) ) &= 2^{n-1} (a \cdot 2^n b + 2^n b \cdot a - 2^{2n} (b \smile_1 b) + 2^{n} (ab-ba)) \\
&= 2^{n-1} (2^{n+1} ab - 2^{2n} (b \smile_1 b) ) \\
&= 2^{2n}(ab-2^{n-1} (b \smile_1 b)).
\end{align*}
Hence by Definition \ref{defn: beta_2} we have
\begin{align*}	
\beta_2(2^{n-1}	x^2) &= [\ol{\frac{1}{2^{2n}} d (2^{n-1} a^2 - 2^{2n-1} (a \smile_1 b))}] \\
&= [\ol{ ab-2^{n-1} (b \smile_1 b)}] \\
&= [xy -[\ol{2^{n-1}(b \smile_1 b)}]].
\end{align*}
We then conclude by using Example \ref{lemma: cup-1} again to identify $[\ol{2^{n-1}(b \smile_1 b)}] = \wt{\Sq}^{2k} y$.
\end{proof}

\begin{thm}\label{thm: pairing identity}
Let $X $ be a smooth projective variety over $\F_q$ of dimension $2d$. For $x \in H^{2d}_{\et}(X;\Z/2^n \Z(d))$, we have 
\[
x \smile \beta (x) = \wt{\Sq}^{2d} (\beta( x)) \in H^{4d+1}_{\et}(X;\Z/2^n \Z(2d)).
\]
\end{thm}

\begin{proof}
By Proposition \ref{prop: beta_2 identity}, we have the identity 
\[
\beta_2(2^{n-1} x^2)= [x \cdot \beta (x) - \wt{\Sq}^{2k} (\beta (x)) ] \in H^{4d+1}_{\et}(X;\Z/2^n\Z(2d))/\Ima \beta.
\]
Therefore we will be done if we can show that the images of $\beta$ and $ \beta_2$ in $H^{4d+1}_{\et}(X;\Z/2^n \Z(2d))$ are both $0$. Since $\beta$ and $\beta_2$ are the reductions of $\wt{\beta}$ and $\wt{\beta}_2$, it suffices to prove the stronger statement that $\wt{\beta}$ and $\wt{\beta}_2$ vanish in the appropriate degree, which is what we shall do. 

Note that the image of $\wt{\beta}$ is automatically $2^n$-torsion. Similarly, from the definition of $\wt{\beta}_2$ it is immediate that its image is $2^{2n}$-torsion. Indeed, referring to Definition \ref{defn: beta_2} we see that $\ell^{2n} \wt{\beta}_2(x) = [d (a -\ell^n b) ] $ is manifestly a coboundary. (In general, $\Ima(\wt{\beta}_r)$ is $2^{rn}$-torsion.) But by Poincar\'{e} duality we have
\[
H^{4d+1}_{\et}(X;\Z_2^n(2d)) \cong \Z_2
\]
is torsion-free, so the images of $\wt{\beta}$ and $\wt{\beta}_2$ in $H^{4d+1}_{\et}(X;\Z/2^n\Z(2d))$ are necessarily $0$. 
\end{proof}

\section{Stiefel-Whitney classes in \'{e}tale cohomology}\label{sec: SW}

Theorem \ref{thm: pairing identity} recasts the pairing $\langle \cdot, \cdot \rangle_n$ in terms of the (generalized) Steenrod squares. But in order for this formula to be useful, we need some way to explicitly calculate the relevant Steenrod operations. In the classical theory of smooth manifolds there is a formula, due originally to Wu, relating the action of certain Steenrod operations as cupping with Stiefel-Whitney classes. This section and the next are concerned with establishing an analogue of this formula in absolute \'{e}tale cohomology for smooth proper varieties over finite fields. The first task, which we take up in this section, is to define an appropriate notion of Stiefel-Whitney classes. 

Much of the material of this section was influenced by \cite{Urabe96}. The definition of the classes $w_i$ already appears in \cite{Urabe96}, though it is phrased in less generality there. The idea to find lifts of the $w_i$ in terms of Chern classes was also inspired by \cite{Urabe96}, as is the proof of Theorem \ref{thm: chern classes}. 

\subsection{Cohomology with supports}\label{subsec: steenrod on relative coh}

Let $i \co X \hookrightarrow Y$ be a closed subscheme. We recall the definition of the ``cohomology of $Y$ with supports in $X$'' \cite[\S I.10]{FK88}.
 Let $j \colon U \hookrightarrow Y$ denote the inclusion of the (open) complement of $X$ in $Y$. Then $H_X^k(Y; \Cal{F})$ is defined to be the $k$th right derived functor of 
\[
\Cal{F} \mapsto \ker (\Cal{F}(Y) \mapsto j_*j^* \Cal{F}(Y)).
\]

%\begin{remark} 
% For $\Cal{F} \in D^b_c(Y)$, we have an exact triangle in $D^b_c(Y)$:
%\[
%i_! i^! \Cal{F} \rightarrow \Cal{F} \rightarrow Rj_* j^* \Cal{F}.
%\]
%Then the homology groups of $R\Gamma_X(Y;\Cal{F}) := R\Gamma(Y;i_!i^! \Cal{F})$ are the $H_X^k(Y; \Cal{F})$. This makes it clear that the $H_X^k(Y; \Cal{F})$ fit into a long exact sequence. 
%\end{remark}

We want to define Steenrod operations on $H_X^k(Y; \Cal{F})$. Since we have been in the habit of defining Steenrod operations via \'{e}tale homotopy theory, we need to realize the cohomology with supports in terms  of \'{e}tale homotopy theory, as the cohomology of a certain pro-space. This is explained in \cite[\S 14]{Fr82}. The key features of this construction are summarized below.  

\begin{defn}
We define $\Et_X(Y) $ to be the mapping cylinder of $\Et(U) \rightarrow \Et(Y)$ in the sense of \cite[p. 140]{Fr82}. For any locally constant sheaf $\Cal{F}$ on the \'{e}tale site of $Y$, we denote by $\Et(\Cal{F})$ the corresponding local coefficient system on $\Et(Y)$ as in \S  \ref{subsec: etale homotopy}, and also for its pullback to $\Et_X(Y) $ in the sense of \cite[p. 140]{Fr82}. Then we have a canonical identification $H^*(\Et_X(Y); \Et(\Cal{F})) \xrightarrow{\sim} H^*_X(Y; \Cal{F})$ \cite[Proposition 14.3, Corollary 14.5, Proposition 14.6]{Fr82}, and we define the Steenrod operations on $H^*_X(Y; \Cal{F})$ via this identification as in \S \ref{sec: steenrod}.
\end{defn}

\subsection{Construction of \'{e}tale Stiefel-Whitney classes}\label{subsec: SW construction}

Let $k$ be a field of characteristic not equal to $2$. Let $i \colon X \hookrightarrow Y$ be a codimension $r$ closed embedding of smooth varieties over $k$. Then we have a \emph{cycle class} $s_{X/Y} \in H^{2r}_X(Y; \Z_{\ell}(r))$, which can be described as the image of $1$ under the \emph{Gysin isomorphism}   
\begin{equation}\label{Gysin}
\phi \colon H^j(X; \Cal{F}) \xrightarrow{\sim}  H^{j+2r}_X(Y; \Cal{F}(r)),
\end{equation}
which holds for any locally constant constructible sheaf $\Cal{F}$ on $Y$ \cite[\S I.10]{FK88}.

We are going to apply this with $Y$ being the total space of a vector bundle $E$ over $X$, and $i \co X \hookrightarrow E$ being the zero section. 

\begin{defn}
Let $ E $ be a vector bundle over $X$. We define the \emph{$j$th Stiefel-Whitney class} of $E$ by 
\begin{equation}\label{eq: defn w_i}
w_j(E) = \phi^{-1}(\Sq^j s_{X/E} ).
\end{equation}
Define the \emph{total Stiefel-Whitney class} to be $w(E) := \sum w_i(E)$. If no vector bundle is mentioned, then by default we set $w_i := w_i(TX)$ and $w := \sum w_i$. 
\end{defn}

There is a possibly more intuitive way to phrase the equation \eqref{eq: defn w_i}, which we will use later. The Gysin isomorphism \eqref{Gysin} says that $H_X^{*} (E; \Z/2\Z)$ is a free rank one module over $H^*_{\et}(E; \Z/2\Z)$, which can be identified with $H^*_{\et}(X; \Z/2\Z)$ via $\pi^*$ since $E$ is a vector bundle over $X$. Under this identification \eqref{eq: defn w_i} is equivalent to 
\begin{equation}\label{eq: second defn w_i}
\Sq^i (s_{X/E}) = w_i \cdot s_{X/E}. 
\end{equation}

\begin{remark}
The reason that we call these ``Stiefel-Whitney classes'' is that Thom observed in \cite{Thom51} that an exactly analogous construction for manifolds produces the usual Stiefel-Whitney classes.\footnote{Unfortunately, this notation clashes with the established tradition of using the term ``Stiefel-Whitney classes'' to denote the characteristic classes of quadratic bundles. The two definitions coincide for smooth manifolds, but in general there is no relation between them.} The construction goes as follows (a reference is \cite[\S 8]{MS74}). Let $M$ be a topological manifold and $E$ be a vector bundle of rank $r$ over $M$. Let $i \colon M \hookrightarrow E$ denote the inclusion of $M$ as the zero section of $E$. Let $E_0 = E - i(M)$. We have a Thom isomorphism 
\[
\phi \colon H^i(M; \Z/2\Z) \cong  H^{i+2r} (E, E_0; \Z/2\Z),
\]
and $w_i (E) = \phi^{-1} (\Sq^i \phi(1))$.

\end{remark}

\subsection{Steenrod squares of Stiefel-Whitney classes}\label{app: A}

The following technical lemma is needed later in \S \ref{subsec: lifting SW}. The reader may safely skip this subsection for now and refer back to it when necessary. 

\begin{lemma}\label{lem: squares of SW classes}
For any $i$ and $j$, $\Sq^i(w_j)$ can be expressed as a polynomial in the Stiefel-Whitney classes $\{w_l\}$.
\end{lemma}

\begin{remark}
The analogue of Lemma \ref{lem: squares of SW classes} for singular cohomology is immediate from the fact that the ring $H^*(\mrm{BO}(\R); \Z/2\Z)$ is generated by Stiefel-Whitney classes. But because of the way that we have defined the classes $w_i$ in \'{e}tale cohomology, Lemma \ref{lem: squares of SW classes} is not quite obvious. 
\end{remark}

\begin{proof}[Proof of Lemma \ref{lem: squares of SW classes}]
We will use the identities of Steenrod squares from \S \ref{subsec: classical Steenrod}. Note that we may assume that $i <j$, since for $i>j$ we have $\Sq^i (w_j) = 0$ and for $i = j$ we have $\Sq^i(w_j) = w_j^2$.

We induct on $j$, and then (for fixed $j$) on $i$, with the base case $j=0$ being trivial, and the base cases $i=0$ being trivial for any $j$ since $\Sq^0 = \Id$. Consider the expression 
\[
\Sq^i \Sq^j (s_{X/TX}).
\] 
On the one hand, we have by \eqref{eq: second defn w_i} that 
\[
\Sq^i \Sq^j (s_{X/TX}) = \Sq^i(w_j \cdot s_{X/TX}).
\]
By the Cartan formula,
\begin{align*}
\Sq^i(w_j \cdot s_{X/TX}) &= \sum_{k+\ell=i} \Sq^k(w_j) \Sq^{\ell} (s_{X/TX}) \\
&=  \left( \Sq^i(w_j) + \sum_{k<i} \Sq^k (w_j) w_{i-k}   \right) s_{X/TX}.
\end{align*}
By the induction hypothesis, $\Sq^k(w_j)$ is a polynomial in the $\{w_l\}$ for each $k<i$, so the upshot is that
\begin{equation}\label{eq:  double Sq}
\Sq^i \Sq^j (s_{X/TX}) = (\Sq^i(w_j) + \text{Poly}(\{w_l\}))s_{X/TX}.
\end{equation}
From \eqref{eq:  double Sq} it is clearly sufficient to show that $\Sq^i \Sq^j (s_{X/TX})$ is a polynomial in the $\{w_l\}$ times $s_{X/TX}$. For this we use the Adem relations: for $0<i<2j$ we have 
\[
\Sq^i\Sq^j = \sum_{k=0}^{[i/2]} \binom{j-1-k}{i-2k} \Sq^{i+j-k} \Sq^k,
\]
hence 
\[
\Sq^i \Sq^j (s_{X/TX}) = \sum_{k=0}^{[i/2]} \binom{j-1-k}{i-2k} \Sq^{i+j-k} (w_k \cdot s_{X/TX}).
\]
Every index $k$ in this sum is strictly less than $j$ since we assumed $i<j$, so every summand is a polynomial in the $\{w_l\}$ times $s_{X/TX}$ by the induction hypothesis, which is what we wanted.
\end{proof}

\begin{remark}
Since the preceding argument could have been carried out equally well in singular cohomology, the proof makes it clear that our $\Sq^i w_j$ is given by the same formula as in algebraic topology. 
\end{remark}

\subsection{Properties of the Stiefel-Whitney classes}\label{SWproperties}

We now record that the Stiefel-Whitney classes, as constructed in \S \ref{subsec: SW construction}, enjoy the usual properties of topological Stiefel-Whitney classes. 

\begin{enumerate}
\item We have $w_i(E) \in H^i(X; \Z/2\Z)$, with $w_0=1$ and $w_i=0$ for $i> 2 \rank E$. 

\item ({\sc naturality}) If $f \colon X' \rightarrow X$ then 
\[
f^* w_i(E) = w_i(f^* E).
\]
\item ({\sc Whitney product formula}) We have 
\[
w_i(E \oplus E') = \sum_{k=0}^i w_k(E) \smile w_{i-k}(E').
\]
If we set $w = \sum w_i$ to be the total Stiefel-Whitney class, then this can be written more succinctly as
\[
w(E \oplus E') = w(E) \cdot w(E').
\]

\end{enumerate}

\noindent \textbf{Proofs.} 
It is well-known in the topological setting (cf. \cite[\S 8]{MS74}) that the characteristic properties of Stiefel-Whitney classes can be formally derived from those of the Steenrod squares. Since our \'{e}tale Stiefel-Whitney classes are also based on Steenrod operations, essentially the same proofs go through. Nonetheless, we spell them out because they will be used in the proof of Theorem \ref{thm: chern classes} below. 

\begin{enumerate}
\item Immediate from the fact that $\Sq^0 = \Id$ and $\Sq^i$ vanishes on $H^j$ if $i>j$.
\item Immediate from the naturality of the Gysin map and Steenrod squares.
\item We begin by considering a general setup. Suppose we have two closed embeddings of smooth proper varieties
\begin{align*}
i \colon X \hookrightarrow Y & \quad \text{codimension $r$}, \\
i' \colon X' \hookrightarrow Y & \quad \text{codimension $r'$}.
\end{align*}
We consider the two corresponding Gysin maps obtained:
\begin{align*}
\phi \colon H^*(X; \Z/2\Z) & \cong H^{*+r}_X(Y; \Z/2\Z), \\
\phi' \colon H^*(X'; \Z/2\Z) & \cong H^{*+r}_{X'}(Y'; \Z/2\Z).
\end{align*}
These send $\phi(1) = s_{X/Y} $ and $\phi'(1) =s_{X'/Y'}$. By the compatibility of the Gysin map for products, we have that for the closed embedding $X \times X' \hookrightarrow Y \times Y$, the Gysin isomorphism 
\[
\phi \smile \phi' \colon H^*(X \times X'; \Z/2\Z) \cong  H^{*+r+r'}_{X \times X'}(Y \times Y'; \Z/2\Z) 
\]
sends $1 \mapsto s_{X/Y}  \smile s_{X'/Y'}$. Now taking $Y$ and $Y'$ to be the total spaces of $E$ and $E'$, and applying the Cartan formula of $\Sq$ and the definition of Stiefel-Whitney classes, we obtain
\begin{align*}
w_{X \times X'} \smile (s_{X/Y} \smile s_{X'/Y'})  &= \Sq (s_{X/Y} \smile s_{X'/Y'}) \\
& = \Sq  (s_{X/Y})  \smile \Sq (s_{X'/Y'} ) \\ 
&= (w_X \smile s_{X/Y}) \smile (w_{X'} \smile s_{X'/Y'}) \\
&= (w_X \smile w_{X'}) s_{X/Y} \smile s_{X'/Y'}.
\end{align*}
Finally, pulling back via the diagonal $\Delta \colon X \hookrightarrow X \times X$ and using naturality gives the result. 

\end{enumerate}

It is formal that the Whitney product formula for direct sums implies it for extensions: 

\begin{lemma}
If 
\[
0 \rightarrow E' \rightarrow E \rightarrow E'' \rightarrow 0
\]
is a short exact sequence of vector bundles on $X$, then 
\[
w(E) = w(E') \smile  w(E'').
\]
\end{lemma}

\begin{proof}
The proof is the same as for \cite[Lemma 2.7]{Urabe96}.

%%%%%%%%%%COMMENTED OUT%%%%%%%%%
\begin{comment}
 but we reproduce the argument for the convenience of the reader. Choose a trivializing open cover $U_i$ on $X$ for $E$. The extension structure means that the transition functions for $E$ are upper ``block-triangular''
\[
\tau(E)_{ij} = \begin{pmatrix} \tau(E')_{ij} & u_{ij} \\ & \tau(E'')_{ij} \end{pmatrix}
\]
where $\tau(E')_{ij}$ are the transition functions for $E'$, and $\tau(E'')_{ij}$ are those for $E''$. Form the vector bundle $\Cal{E}$ over $X \times \A^1$ by gluing the trivial bundles over the open cover $U_i \times \A^1$ by the transition functions 
\[
\tau(E)_{ij} = \begin{pmatrix} \tau(E')_{ij} & tu_{ij} \\ & \tau(E'')_{ij} \end{pmatrix}
\]
where $t$ is the coordinate on $\A^1$.

Let $s_0 \colon X \rightarrow X \times \A^1$ be the embedding $x \mapsto (x,0)$ and $s_1 \colon X \rightarrow X \times\A^1$ be the embedding $x \mapsto (x,1)$. Then $s_1^* \Cal{E} \cong E$, while $s_0^* \Cal{E} \cong E' \oplus E''$. By naturality of Stiefel-Whitney classes, we have $w(E) = s_1^* w(\Cal{E})$ and $w(E' \oplus E'') = s_0^* w(\Cal{E})$. But $s_0^*$ and $s_1^*$ are both inverse to the isomorphism on cohomology induced by the projection $X \times \A^1 \rightarrow X$, so they coincide. The result then follows from the split case. 
\end{comment}
%%%%%%%%%%COMMENTED OUT%%%%%%%%%
\end{proof}

\subsection{Lifting Stiefel-Whitney classes to integral cohomology} 

We shall see in \S \ref{sec: alternating} that it is crucial to know whether our Stiefel-Whitney classes lift to integral cohomology. The goal of this subsection  is to prove Theorem \ref{thm: chern classes}, which answers this question. 

Our first task is to address a technical subtlety that will come up in the proof of Theorem \ref{thm: chern classes}. There are the two short exact sequences
\begin{equation}\label{eq: Z/2 SES 1}
0 \rightarrow \Z/2\Z \rightarrow \Z/4\Z \rightarrow \Z/2\Z \rightarrow 0
\end{equation}
and 
\begin{equation}\label{eq: Z/2 SES 2}
0 \rightarrow \mu_2 \rightarrow \mu_4 \rightarrow \mu_2 \rightarrow 0.
\end{equation}
Since $\mu_2$ is canonically identified with $\Z/2\Z$, as we are not in characteristic $2$, both sequences induce Bockstein operations $H^*_{\et}(X; \Z/2\Z) \rightarrow H^{*+1}_{\et}(X; \Z/2\Z)$, but they are \emph{not} necessarily the same. In \S \ref{subsec: classical Steenrod} we noted that the Bockstein operation for \eqref{eq: Z/2 SES 1} is $\Sq^1$. Let us denote by $\beta^{(1)}$ the Bockstein operation for \eqref{eq: Z/2 SES 2}. We need to quantify the difference between these two operations. For this discussion, it will help to maintain a psychological distinction between $\mu_2$ and $\Z/2\Z$. 

\begin{lemma}\label{lemma: boundary difference}
Let $\alpha$ be the image of $1 \in H^0_{\et}(X; \Z/2\Z) \xrightarrow{\sim} H^0_{\et}(X; \mu_2)$ under the boundary map $\beta^{(1)}$. Then for all $c \in H^*_{\et}(X; \Z/2\Z) \xrightarrow{\sim} H^*_{\et}(X; \mu_2)$ we have
\[
\beta^{(1)}(c) = \Sq^1(c) +  \alpha \smile c.
\]
\end{lemma}

\begin{proof}
Since $\mu_4$ is a module over $\Z/4 \Z$, the cohomology $H_{\et}^*(X; \mu_4)$ is a module over $H_{\et}^*(X; \Z/4\Z)$. We similarly view $H_{\et}^*(X; \mu_2) $ as a module over $H_{\et}^*(X; \Z/2\Z)$. 

The reduction map $\mu_4 \rightarrow \mu_2$,
viewed as part of the short exact sequence \eqref{eq: Z/2 SES 2}, is compatible with the reduction map $\Z/4 \Z \rightarrow \Z/2\Z$, viewed as part of \eqref{eq: Z/2 SES 1},
for the respective module structures. Hence the induced maps on cohomology satisfy the same compatibility: the reduction map  
\[
H^*_{\et}(X; \mu_4)  \rightarrow  H^*_{\et}(X;\mu_2)
\]
is compatible as a map of modules with respect to the map of rings 
\[
H^*_{\et}(X; \Z/4 \Z)  \rightarrow  H^*_{\et}(X;\Z/2\Z).
\]
More precisely, let $x \in H^i_{\et}(X; \mu_2)$ and $r \in H^j_{\et}(X; \Z/2\Z)$,
so that $x$ is viewed as a module element and $r$ is viewed as a ring element. Then $rx \in H^{i+j}_{\et}(X; \mu_2)$, and we are saying that
\begin{equation}\label{eq: boundary_module}
\beta^{(1)} (rx) = (\Sq^1 r)  x + r \beta^{(1)}(x).
\end{equation}
This is seen immediately upon going back to the definition of the boundary map, using that the coboundary map on cochains is a derivation. 
%At the level of cochains, this boils down to the fact that the coboundary map for $H^*_{\et}(X; \mu_4)$ is compatible with the coboundary map for $H^*_{\et}(X; \Z/4\Z)$. To see this explicitly, take cochain representatives $c_x \in C^*_{\et}(X; \mu_2)$ and $c_r \in C^*_{\et}(X; \Z/2\Z)$ for $x$ and $r$, and lift them to representatives $\wt{c_x} \in C^*_{\et}(X; \mu_4)$ and $\wt{c_r} \in C^*_{\et}(X; \Z/4\Z)$. Then it is immediate from the definition of the coboundary map that 
%\[
%d^{(1)}(\wt{c_r} \wt{c_x} ) = d(\wt{c_r})  \wt{c_x} + \wt{c_r} d^{(1)}(\wt{c_x}) 
%\]
%where $d^{(1)}$ and $d$ are the coboundary maps for $C^*_{\et}(X; \mu_4)$ and $C^*_{\et}(X; \Z/4\Z)$, respectively. The map $rx \mapsto [d(\wt{c_r}\wt{c_x})]$ is precisely the coboundary map in the long exact sequence of cohomology, which establishes \eqref{eq: boundary_module}.

The lemma then follows from taking $r=c$ and $x=1 \in H^0(X; \mu_2)$ in \eqref{eq: boundary_module}.

\end{proof}

\begin{remark}\label{rem: alpha}
The element $\alpha \in H^1_{\et}(X; \mu_2)$ is actually the pullback of a universal $\alpha \in H^1_{\et}(\Spec \F_q; \mu_2)$ which vanishes if and only if $q \equiv 1 \pmod{4}$. Indeed, \eqref{eq: Z/2 SES 1} and  \eqref{eq: Z/2 SES 2} are obviously the same for $q \equiv 1 \pmod{4}$. We also note for later use that $\alpha$ lifts to $H^1_{\et}(X; \Z_2(1))$, because $\beta^{(1)}$ is the reduction of the Bockstein for 
\[
0 \rightarrow \Z_2(1) \xrightarrow{2} \Z_2(1) \rightarrow \mu_2  \rightarrow 0.
\]
\end{remark}

\begin{thm}\label{thm: chern classes}
Let $X$ be a smooth variety over a finite field $\F_q$ of characteristic not $2$ and $E$ a vector bundle on $X$ of rank $r$. Let $\alpha$ be as in Lemma \ref{lemma: boundary difference}. Then we have:
\[
w(E) := 1+w_1 + w_2 + \ldots + w_{2r} =  \begin{cases} (1+\alpha) \ol{c}_{\even} + \ol{c}_{\odd} & r \text{ odd}, \\ \ol{c}_{\even} + (1+\alpha) \ol{c}_{\odd} & r \text{ even}, \end{cases}
\]
where 
\begin{align*}
c_{\even} &= 1 + c_2 + c_4 + \ldots  \in H^*_{\et}(X;\bigoplus_{i\in \Z} \Z_2(i))\\
c_{\odd} &= c_1 + c_3 + \ldots  \in H^*_{\et}(X;\bigoplus_{i\in \Z} \Z_2(i))
\end{align*}
and $\ol{c}$ means the reduction of $c$ modulo $2$. 
\end{thm}

\begin{proof}
Grothendieck showed \cite{Gro58} that the definition of all characteristic classes can be obtained from the axioms in \S \ref{SWproperties} plus the definition of the characteristic classes for arbitrary line bundles. Therefore, it suffices to check that the formula above satisfies the properties in \S \ref{SWproperties} and is correct for all line bundles. 

The fact that it satisfies axiom (1) of \S \ref{SWproperties} is evident from the definition. The fact that it satisfies (2) is immediate from the observation that the Chern classes satisfy the Whitney sum and naturality property. The fact that it satisfies (3) also follows from the analogous property of Chern classes plus a case analysis of the formula. For example, when summing two bundles $\Cal{E}$ and $\Cal{E}'$ of odd rank with Chern classes $c$ and $c'$, the product of the classes claimed in the formula is 
\begin{align*}
((1+\alpha)\ol{c}_{\even} + \ol{c}_{\odd}) ((1+\alpha)\ol{c'}_{\even}+ \ol{c'}_{\odd}) &= (\ol{c}_{\even} \ol{c'}_{\even} + \ol{c}_{\odd}  \ol{c'}_{\odd} ) \\
& \qquad  + (1+\alpha) (\ol{c}_{\even} \ol{c'}_{\odd} + \ol{c}_{\odd} \ol{c'}_{\even})
\end{align*}
because $(1+\alpha)^2 = 1$, and the Whitney sum formula for Chern classes implies that the right hand side is indeed $c_{\even}(\Cal{E} \oplus \Cal{E}') + (1+\alpha) c_{\odd}(\Cal{E} \oplus \Cal{E}') $.

Finally we must check the formula for line bundles. What makes this possible is that we only have to verify the formula for $w_1$ and $w_2$, since the higher Stiefel-Whitney classes vanish for degree reasons. Thus we only need to compute $\Sq^1$ and $\Sq^2$, and we have ``explicit'' descriptions of the Steenrod operations on degree $2$ elements in these cases (\S \ref{subsec: classical Steenrod}).

Let $Y$ be the total space of a line bundle $\Cal{L}$ on $X$. We view $X$ as embedded in $Y$ via the zero section, and identify their \'{e}tale cohomology groups via pullback for the projection map $\pi \co Y \rightarrow X$. \\

\noindent \textbf{Calculation of $w_1$.} Recall from \eqref{eq: second defn w_i} that $w_1$ is defined by
\[
\Sq^1 s_{X/Y} = w_1 \smile s_{X/Y}.
\]
But the cycle class $s_{X/Y}$ lifts compatibly to $H_X^{2}(Y; \mu_{2^j})$ for all $j$, hence even to $H_Y^{2}(X; \Z_{2}(1))$ (cf. \cite[\S II.2]{FK88}). In particular $s_{X/Y}$ lies in the image of the reduction map $H^i_X(Y; \mu_4 ) \rightarrow H^i_X(Y; \mu_2)$.  The long exact sequence for \eqref{eq: Z/2 SES 2} then shows that $\beta^{(1)}(s_{X/Y}) = 0$. We are really interested in the \emph{other} boundary map $\Sq^1$, but Lemma \ref{lemma: boundary difference} tells us the difference between them:
\[
\Sq^1 (s_{X/Y})= \beta^{(1)} (s_{X/Y} )+ \alpha \smile s_{X/Y} = \alpha \smile s_{X/Y}.
\]
Hence $\alpha = w_1$, as required. \\

\noindent \textbf{Calculation of $w_2$.} The argument is essentially the same as in the proof of \cite[Lemma 2.6]{Urabe96}. Again, \eqref{eq: second defn w_i} tells us that
\[
\Sq^2 s_{X/Y} = w_2 \smile s_{X/Y} \in H^4_X(Y; \Z/2\Z). 
\]
Since $s_{X/Y} \in H^2_X(Y; \Z/2\Z)$ we have that $\Sq^2s_{X/Y} = s_{X/Y} \smile s_{X/Y}$ (using one of the explicit ``special cases'' from \S \ref{subsec: classical Steenrod}).

We now need to recall a property of the cycle class, which is a special case of a more general discussion to come in \S \ref{step2}. If $X \hookrightarrow Y$ is a codimension 1 closed embedding of smooth varieties, then we have a cycle class $\mrm{cl}_Y(X) \in H^2(Y; \mu_2)$ which is the image of the line bundle $\Cal{O}_Y(X)$ under the Chern class map $H^1(Y; \G_m) \rightarrow H^2(Y; \mu_2)$. This class $\mrm{cl}_Y(X)$ also coincides with the image of $s_{X/Y}$ under the natural map $H_X^*(Y) \rightarrow H_{\et}^*(Y)$. (A reference is \cite[Proposition II.2.2 and Proposition II.2.6]{FK88}\footnote{The book \cite{FK88} makes a blanket assumption that the ground field is separably closed, but the proofs of these particular facts don't require this assumption.}.)

Consider the commutative diagram 
\[
\xymatrix{
\Pic(X)  = H^1_{\et}(X; \G_m) \ar[r] \ar[d]_{\pi^*} & H^2_{\et}(X; \Z/2\Z) \ar[d]_{\pi^*} \\
\Pic(Y) = H^1_{\et}(Y; \G_m) \ar[r] & H^2_{\et}(Y; \Z/2\Z)
}
\]
An elementary calculation shows that the line bundle $\Cal{L}$ on $X$ whose total space is $Y$ pulls back to $\Cal{O}_Y(X)$ on $Y$, i.e. the line bundle associated to the divisor of the zero-section in $Y$. The upshot is that in $H^4_X(Y; \Z/2\Z)$, we have 
\begin{align*}
\ol{c_1}(\Cal{O}_Y(X)) \smile s_{X/Y} & =\mrm{cl}_Y(X) \smile s_{X/Y} = s_{X/Y} \smile s_{X/Y} = \Sq^2 (s_{X/Y} ),
\end{align*}
which shows that $w_2 = \ol{c_1}(\Cal{O}_Y(X)) \in H^2_{\et}(Y; \Z/2\Z)$. Since we have already established that $\Cal{L}$ pulls back to $\Cal{O}_Y(X)$ under the projection map $\pi \colon Y \rightarrow X$, naturality for Chern classes and the fact that $\pi^*$ induces an isomorphism on cohomology shows that $w_2 = \ol{c_1}(\Cal{L}) \in H_{\et}^2(X; \Z/2\Z)$. 
\end{proof}

\section{A Wu Theorem for \'{e}tale cohomology}\label{sec: Wu}

Now we relate the Stiefel-Whitney classes just constructed in \S\ref{sec: SW} with Steenrod operations. \emph{In this section it is understood that all cohomology is with $\Z/2\Z$-coefficients}, so we may suppress it from our notation.

\subsection{Wu's Theorem for smooth manifolds}

We first explain the classical version of Wu's theorem. Let $M$ be a closed smooth manifold of dimension $n$, so that the cup product induces a perfect duality on $H^*(M)$. Then for a cohomology class $x \in H^{n-i}(M; \Z/2\Z)$ the map $x \mapsto \Sq^i x \in H^n(M; \Z/2\Z)$ must, by Poincar\'{e} duality, be represented by a class $v_i \in H^i(M; \Z/2\Z)$, i.e.
\[
v_i \smile x  = \Sq^i x \text{ for \emph{all} }  x \in H^i(M; \Z/2\Z).
\]
This $v_i$ is called the $i$th \emph{Wu class}. 

Let $v := \sum_i v_i$ be the total Wu class and $w := \sum w_i$ be the total Stiefel-Whitney class of $TM$. Then Wu's formula relates the two in the following way:

\begin{thm}[Wu]
We have $w = \Sq v$. 
\end{thm}

\begin{remark}
Note that $\Sq$ is invertible, so Wu's Theorem completely describes $v$ in terms of $w$. 
\end{remark}

\begin{example}\label{ex: wu classes}
We use Wu's Theorem to calculate a few small examples. Equating terms of degree $1$, we deduce that 
\[
v_1 =  w_1.
\]
Equating terms of degree $2$, we deduce that $v_2 + \Sq^1 v_1 = w_2$, which we can rewrite as 
\[
v_2 = w_2 + w_1^2.
\] 
\end{example}

\subsection{A Wu Theorem for varieties over finite fields}

The aim of this section is to prove a version of Wu's Theorem in the setting of \'{e}tale cohomology. For varieties over \emph{separably closed} fields, this is done in \cite[Theorem 0.5]{Urabe96}. In that case one can more or less transpose the usual proof for manifolds, essentially because the $\ell$-adic cohomology of smooth varieties over separably closed fields behaves very similarly to the singular cohomology of complex manifolds. In particular, for a surface that lifts to characteristic $0$, the classical version of Wu's Theorem implies the version for geometric $\ell$-adic cohomology. The main result of this section (Theorem \ref{Wu theorem finite field}) is that the same formula also holds for \emph{absolute} \'{e}tale cohomology over \emph{finite} fields, with our definitions of the $w_i$ from \S \ref{sec: SW}. Because the ground field is not separably closed there are some significant new difficulties; one indication of this is that the proof requires \'{e}tale homotopy theory.

\begin{remark}\label{rem: philosophy} The author has come to think about this philosophically as follows. A major defect in the analogy between varieties over $\F_q$ and topological spaces fibered over $S^1$ is that in the latter situation one can forget the fibration and consider the bare topological space, while there is no corresponding move for varieties over $\F_q$. Thus any operation performed in the category of varieties over $\F_q$ is really a ``relative'' operation: the product of varieties over $\F_q$ corresponds to the \emph{fibered} product of manifolds over $S^1$, the tangent bundle of a variety over $\F_q$ corresponds to the \emph{relative} tangent bundle over $S^1$, etc. Because of this, there are some steps in the proof of Wu's Theorem that have no analogue in the category of varieties over $\F_q$. However, passing to \'{e}tale homotopy type allows one to disassociate a variety from this fibration, and thus acquire some of the additional flexibility enjoyed by topological spaces. 
\end{remark} 

\begin{thm}\label{Wu theorem finite field} Let $X$ be a smooth, proper, geometrically connected variety over $\F_q$. Define the Wu class $v \in H^*_{\et}(X; \Z/2\Z)$ to be the unique cohomology class such that 
\[
\int \Sq x = \int v \smile x \text{  for all $x \in H^*_{\et}(X; \Z/2\Z)$}.
\]
 Then we have $w = \Sq v$.
\end{thm}

The reader is recommended to skip the proof on the first pass through the paper, as it is quite lengthy and nothing but the statement of Theorem \ref{Wu theorem finite field} will be used in the sequel.

\subsection{Proof of Theorem \ref{Wu theorem finite field}}\label{WuProof}

Our proof of Theorem \ref{Wu theorem finite field} proceeds in four steps. Steps 2 and 3 are essentially a translation of the usual (topological) proof to algebraic geometry. Step 1 performs a technical reduction to the case where the topological argument begins, and is necessary because of the lack of ``tubular neighborhoods'' in algebraic geometry. Finally, Step 4 bridges a new technical difficulty, the spirit of which is discussed in Remark \ref{rem: philosophy}, that arises here because our ground field is not separably closed.

\subsubsection{Step 1} Recall from \eqref{eq: second defn w_i} that we defined the Stiefel-Whitney classes $w_i$ by 
\[
\Sq^i s_{X/TX} =  \pi^*(w_i) \smile s_{X/TX}
\]
where $\pi \co TX \rightarrow X$ is the projection. Recall also that the normal bundle of $X$ in $X \times X$ is isomorphic to $TX$. The purpose of this step is to prove the following Lemma, which is motivated by the preceding considerations.

\begin{lemma}\label{lem: step 1}
Let $s_{X/X\times X} \in H_X^{*+2n}(X\times X)$  be the image of $1 \in H^0_{\et}(X)$ under the Gysin isomorphism 
\[
H_{\et}^*(X) \xrightarrow{\sim} H_X^{*+2n}(X\times X).
\]
Then we have 
\begin{equation}\label{WuProof1}
\Sq^i s_{X/X\times X} =  \mrm{pr}_1^*(w_i) \smile s_{X/X\times X}.
\end{equation}
where $\mrm{pr}_1 \co X \times X\rightarrow X$ denotes projection to the first factor.
\end{lemma}

From the definitions Lemma \ref{lem: step 1} is an immediate consequence of the following Lemma. 

\begin{lemma}
Let $X \hookrightarrow Y$ be a codimension $n$ closed embedding of smooth varieties (over any field) and let
\begin{align*}
\phi_1 \colon H^*(X) & \xrightarrow{\sim} H^{*+2n}_X(N_{X/Y})
\phi_2 \colon H^*(X) & \xrightarrow{\sim} H^{*+2n}_X(Y) \\
\end{align*}
be the two Gysin isomorphisms. Then 
 \[
w(N_{X/Y}) := \phi_1^{-1} (\Sq s_{X/N_{X/Y}}) = \phi_2^{-1} (\Sq s_{X/Y}).
\]
\end{lemma}

\begin{remark}\label{rem: excision} If $X$ were a smooth manifold, we could argue directly since we have an isomorphism 
\[
H_X^{*+2n}(N_{X/Y}) \cong H^{*+2n}(U, U \setminus X) 
\]
where $U$ is a tubular neighborhood of the zero-section in $Y$, and we have 
\[
H^{*+2n}(U, U \setminus X)  \cong H^{*+2n}_X(Y) :=  H^{*+2n}(Y, Y \setminus X)
\]
by excision. Since these isomorphisms are pullbacks induced by maps of spaces, they commute with Steenrod squares. A referee has suggested that an analogue of this argument may be carried out in our setting using the Morel-Voevodsky Purity Theorem. 
\end{remark}

\begin{proof}
The key fact is that if $X \hookrightarrow Y$ is a closed embedding, then there is a flat family deforming the inclusion $X \hookrightarrow Y$ into the zero-section $X \hookrightarrow N_{X/Y}$ (``deformation to the normal cone''). This allows us to carry out the idea of Remark \ref{rem: excision}.

More precisely, there is a flat family $\Cal{Y} \rightarrow \A^1$ which restricts to the trivial family away from the origin, $\Cal{Y}|_{\A^1-0} \cong Y \times (\A^1-0)$, but such that $\Cal{Y}|_0 \cong N_{X/Y}$. Furthermore, there is a closed embedding $X \times \A^1 \hookrightarrow \Cal{Y}$ which restricts to the given embedding $X \hookrightarrow Y$ away from $0$, and $X \hookrightarrow N_{X/Y}$ at $0$. For the construction and proofs of the properties, see \cite[\S 5]{Ful98}. The situation is depicted in the diagram below: 
\[
\begin{tikzcd}
X \ar[d, hook] \ar[r, hook] & X \times \A^1 \ar[d, hook] & X \ar[l, hook'] \ar[d,hook] \\
N_{X/Y} \ar[d] \ar[r, hook] & \Cal{Y} \ar[d]  & Y \ar[l, hook'] \ar[d] \\
0 \ar[r, hook] & \A^1 & t \ar[l, hook'] 
\end{tikzcd}
\]
Applying the Gysin morphism to $X \times \A^1 \hookrightarrow \Cal{Y}$, we have an isomorphism
\[
H^*_{\et}(X \times \A^1 ) \xrightarrow{\sim} H^{*+2n}_{X\times \A^1}(\Cal{Y})
\]
sending $1 \mapsto s_{X \times \A^1/\Cal{Y}}$. Note that $X \times \A^1$ and $Y$ (viewed as the fiber over $t$) intersect transversely in $\Cal{Y}$, and similarly  $X \times \A^1$ and $N_{X/Y}$. Hence from the diagram above we obtain a diagram of maps in cohomology (where the vertical maps are the respective Gysin isomorphisms):
\[
\xymatrix{
H^*_{\et}(X) \ar[d]_{\sim} & \ar[l]  H^*_{\et}(X \times \A^1) \ar[d]_{\sim} \ar[r]  & H^*_{\et}(X) \ar[d]_{\sim} \\
H^{*+2n}_X (N_{X/Y} ) & \ar[l]  H^{*+2n}_{X\times \A^1}(\Cal{Y})  \ar[r] & H^{*+2n}_X(Y) 
}
\]
Under this diagram the Thom classes are mapped as follows, by compatibility with base change (cf. \S 2 of Deligne's expos\'{e} ``Cycle'' in \cite{SGA4_5})
\[
\xymatrix{
1 \ar[d] & \ar[l]  1 \ar[d] \ar[r]  & 1\ar[d] \\
s_{X/N_{X/Y}}  & \ar[l]  s_{X \times \A^1/\Cal{Y}}  \ar[r] & s_{X/Y}
}
\]
Since the horizontal maps in the bottom row are pullbacks they are compatible with $\Sq$, hence send
\[
\xymatrix{
\Sq^i (s_{X/N_{X/Y}})  \ar@{=} & \ar[l]  \Sq^i(s_{X \times \A^1/\Cal{Y}})  \ar[r]  \ar@{=} & \Sq^i(s_{X/Y}).  \ar@{=} \\
}
\]
By definition $\Sq^i (s_{X/N_{X/Y}}) = \pi^*(w_i) \smile s_{X/N_{X/Y}}$, and since the maps 
\[
\xymatrix{
H^{*+2n}_X ( N_{X/Y}) & \ar[l]  H^{*+2n}_{X\times \A^1}(\Cal{Y})  \ar[r] & H^{*+2n}_X(Y) 
}
\]
are isomorphisms of $H^*_{\et}(X) \cong H^*_{\et}(X \times \A^1)$-modules, they send 
\[
\xymatrix{
\Sq^i (s_{X/N_{X/Y}})  \ar@{=}[d] & \ar[l]  \Sq^i(s_{X \times \A^1/\Cal{Y}})  \ar[r]  \ar@{=}[d]  & \Sq^i(s_{X/Y})  \ar@{=}[d] \\
\mrm{pr}_1^*(w_i) \smile s_{X/N_{X/Y}} & \ar[l] \mrm{pr}_1^*(w_i) \smile s_{X \times \A^1/\Cal{Y}}  \ar[r] & \pi^*(w_i) \smile s_{X/Y}
}
\]
as desired. 
\end{proof}

\subsubsection{Step 2}\label{step2} For a regular embedding $X \hookrightarrow Y$, there is a attached \emph{cycle class} $\mrm{cl}_Y(X) \in H^*(Y)$, which in topology would be the ``Poincar\'{e} dual to the fundamental class of $X$ in homology''. The goal of this step is to prove the following lemma.

\begin{lemma}\label{lem: step 2} Let $\Delta := \mrm{cl}_{X \times X} (X) \in H^*_{\et}(X \times X)$ be the cycle class for the diagonal embedding of $X$. Then we have $w = (\mrm{pr}_1)_* (\Sq \Delta) \in H^*_{\et}(X)$.
\end{lemma}

 We first review the definition of the pushforward in cohomology for a map of smooth proper varieties, and then the definition of the cycle class. 
 
\begin{defn}[Pushforwards in cohomology]\label{defn: pushforward variety} If $f \colon X \rightarrow Y$ is a map of smooth proper varieties over $\F_q$ of dimensions $m$ and $n$ respectively, then the pullback map 
\[
f^* \colon H^*_{\et}(Y) \rightarrow H^*_{\et}(X)
\]
induces an adjoint map on the $\Z/2\Z$-dual spaces:
\begin{equation}\label{eq: dual to pull}
(f^*)^{\vee} \co  H^*_{\et}(X)^{\vee} \rightarrow H^*_{\et}(X)^{\vee}.
\end{equation}
We can identify $H^{*}_{\et}(X)^{\vee} \cong H^{2m+1-*}_{\et}(X)$ by Poincar\'{e} duality, obtaining from \eqref{eq: dual to pull} a map 
\[
f_* \colon H^*_{\et}(X) \rightarrow H^{*+2n-2m}_{\et}(Y).
\] 
In particular, we define the \emph{cycle class of $X$ in $Y$} to be $f_*(1) =: \mrm{cl}_Y(X)$. Unwrapping the definition, the class $\mrm{cl}_Y(X)$ is characterized by the identity
\[
\int_X f^* \gamma = \int_Y \mrm{cl}_Y(X) \smile \gamma \quad \text{
for all $\gamma \in H^*(Y)$. }
\]
\end{defn}

We recall some basic properties of this pushforward. The proofs are all immediate from the definition except the last, which is \cite[Proposition 2.7]{FK88}.
\begin{itemize}
\item It is functorial. 
\item We have the \emph{product formula}
\begin{equation}\label{eq: product formula}
f_*(\alpha \smile f^* \beta) = (f_* \alpha) \smile \beta.
\end{equation}
%\item Any map $\Spec \F_q \rightarrow X$ induces a pushforward $H^*_{\et}(\Spec \F_q) \rightarrow H^*_{\et}(X)$ and the image of the generator $\mu \in H^1_{\et}(\Spec \F_q)$ in $H_{\et}^{2m+1}(X)$ is the fundamental class (i.e. generator) $\mu_X$. 
\item If $X \hookrightarrow Y$ is a closed embedding, then the map $H^*_X(Y) \rightarrow H^*_{\et}(Y)$ sends $s_{X/Y} \mapsto \mrm{cl}_Y(X)$.
\end{itemize}

\begin{proof}[Proof of Lemma \ref{lem: step 2}] We now apply the preceding discussion to the case $Y=X \times X$, with $f$ being the diagonal embedding. By Lemma \ref{lem: step 1} we have  
\[
\Sq^i s_{X/X \times X} = \pr_1^*(w_i) \smile s_{X/X\times X} \in H^{2n+i}_X(X \times X). 
\]
Since the map $H^*_X(X \times X) \rightarrow H_{\et}^*(X \times X)$ is induced by a pullback (\S \ref{subsec: steenrod on relative coh}) it is automatically compatible with Steenrod operations, so it sends
\[
\xymatrix{
\Sq^i s_{X/X \times X} \ar[r] \ar@{=}[d] &  \Sq^i \mrm{cl}_{X \times X}(X) \ar@{=}[d] \\
\mrm{pr}_1^*(w_i) \smile s_{X/X \times X}  \ar[r] & \mrm{pr}_1^*(w_i) \smile \mrm{cl}_{X \times X}(X) 
}
\]
Hence by Definition \ref{defn: pushforward variety} and \eqref{eq: product formula} we have
\[
(\mrm{pr}_1)_* (\mrm{pr}_1^*(w_i) \smile \mrm{cl}_{X \times X}(X) ) = w_i \smile (\mrm{pr}_1)_* f_* 1 = w_i.
\]
\end{proof}

\subsubsection{Step 3}

At this point, the classical proof of Wu's theorem proceeds by computing $(\pr_1)_* \mrm{cl}_{X \times X}(X) $ in a second way which is predicated upon the K\"{u}nneth formula 
\[
H^*_{\et}(X \times X) \cong H^*_{\et}(X) \otimes H^*_{\et}(X),
\]
which unfortunately breaks down in our situation. To explain how to repair the argument, we need to make some observations. Note that $H^*_{\et}(X \times X)$ acts by correspondences on $H^*_{\et}(X)$, inducing the map 
\begin{equation}\label{correspondences}
H^*_{\et}(X \times X) \rightarrow \End  (H^*_{\et}(X))
\end{equation}
given explicitly by sending $x \in H^*_{\et}(X \times X)$ to the endomorphism
\begin{equation}\label{eq: correspondence}
\gamma \mapsto (\pr_1)_* ( x \smile \pr_2^* \gamma ).
\end{equation}

\begin{lemma}\label{diag_is_identity}
Let $\Delta := \cla_{X \times X}(X) \in H^{2 \dim X}_{\et}(X \times X)$. Then the map \eqref{correspondences} sends $\Delta \mapsto \Id$. 
\end{lemma}

\begin{proof}
Let $f \co X\hookrightarrow X\times X$ denote the inclusion of the diagonal. Applying \eqref{eq: correspondence} to $x =\Delta$ yields the endomorphism
\[
\gamma \mapsto (\pr_1)_* (f_*(1) \smile p_2^* \gamma)  = (\pr_1)_* f_* (1 \smile f^* \pr_2^* \gamma) .
\]
But since $\pr_1 \circ f = \pr_2 \circ f = \Id$, this last expression is just $\gamma$ again. 
\end{proof}

The map \eqref{correspondences} can be interpreted as a ``pushforward'' in the following way. The projection morphisms $\pr_1, \pr_2 \co X \times X \rightarrow X$ induce maps $\pr_1^*, \pr_2^* \colon H^*_{\et}(X) \rightarrow H^*_{\et}(X \times X)$. From this we get a pullback map 
\[
H^*_{\et}(X) \otimes H^*_{\et}(X) \xrightarrow{\pr_1^* \smile  \pr_2^*} H^*_{\et}(X \times X).
\]
Therefore, we get a dual map in the opposite direction 
\[
H^*_{\et}(X \times X)^{\vee} \rightarrow H^*_{\et}(X)^{\vee} \otimes H^*_{\et}(X)^{\vee}.
\]
Each of these groups is canonically self-dual via Poincar\'{e} duality, so we can identify this with a map 
\begin{equation}\label{push_end}
\varphi_* \colon H^*_{\et}(X \times X) \rightarrow H^*_{\et}(X) \otimes H^*_{\et}(X).
\end{equation}
Note that this map \emph{increases the total degree} by 1. It is a straightforward exercise in unraveling the definitions to see that this map is the same as \eqref{correspondences}, once one makes the appropriate identifications.

Let $(p_1')_*$ and $(p_2')_*$ denote the ``pushforward'' maps 
\[
\begin{tikzcd}
&  \ar[dl, "(p_1')_*"']  H^*_{\et}(X) \otimes H^*_{\et}(X) \ar[dr, "(p_2')_*"]  \\
 H^*_{\et}(X)  & &   H^*_{\et}(X)
 \end{tikzcd}
\]
which are dual to the obvious ``pullbacks''
\[
\begin{tikzcd}
&  H^*_{\et}(X) \otimes H^*_{\et}(X)    \\
 H^*_{\et}(X)  \ar[ur, "(p_1')^*"]  & & H^*_{\et}(X) \ar[ul, "(p_2')^*"']  
 \end{tikzcd}
 \]

\begin{remark}\label{rem: fake variety}
The maps $(p_i')^*$ and $(p_i')_*$ are \emph{not} induced by maps of varieties; indeed $H^*_{\et}(X) \otimes H^*_{\et}(X)$ is not the cohomology of a variety over $\F_q$. However, $H^*_{\et}(X) \otimes H^*_{\et}(X)$ \emph{is} the cohomology of the pro-space $\Et(X) \times \Et(X) $. This means, for instance, that it is equipped with a natural cup product, which is just the tensor product of the cup products on $H^*_{\et}(X)$. Now, $(p_i')^*$ and $(p_i')_*$ \emph{are} induced by maps of pro-spaces, namely the obvious projection maps 
\[
p_i' \co \Et(X) \times \Et(X) \rightarrow \Et(X).
\]
This implies that $(p_i')^*$ and $(p_i')_*$ share the nice formal properties that are enjoyed by all pullbacks and pushforwards: for example, we will use that they satisfy the projection formula, and that $(p_i')^*$ commutes with Steenrod operations. 

However, it is not really necessary to use \'{e}tale homotopy theory to see all this. We can just \emph{formally define} the cup product on $H^*_{\et}(X) \otimes H^*_{\et}(X)$ to be the tensor product of the cup products on $H^*_{\et}(X)$, and \emph{formally define} $\Sq$ on $H^*_{\et}(X) \otimes H^*_{\et}(X)$ to be the tensor product of $\Sq \otimes \Sq$. It is an exercise in elementary algebra to check that this induces a well-defined action of the Steenrod algebra, satisfying  all the axioms of \S \ref{subsec: classical Steenrod}. Similarly, the projection formula for $(p_i)_*$ boils down to a tautology. 
\end{remark}

\begin{lemma}\label{Delta_basis}
Let $X$ be a smooth proper variety over a finite field. Let $\{e_i\}$ be a basis for $H^*_{\et}(X)$ and $\{f_i\}$ the dual basis under Poincar\'{e} duality. Then, letting $\Delta$ be as in Lemma \ref{diag_is_identity}, we have
\begin{equation}\label{eq: push of delta}
\varphi_* \Delta = \sum_i  e_i \otimes  f_i
\end{equation}
where $\varphi_*$ is as in \eqref{push_end}.
\end{lemma}

\begin{proof}
Lemma \ref{diag_is_identity} says that the action of $\Delta$ induced on $H^*(X)$ by \eqref{eq: correspondence} is just the identity map. Therefore, it suffices to show that the  right hand side of \eqref{eq: push of delta} acts as the identity on $H^*(X)$, but this is just a straightforward linear algebra exercise about dual bases. 
\end{proof}

Since the pullback $H^*_{\et}(X) \xrightarrow{\pr_1^*} H^*_{\et}(X \times X)$ obviously factors through 
\[
H^*_{\et}(X) \xrightarrow{(p_1')^* } H^*_{\et}(X) \otimes H^*_{\et}(X)  \xrightarrow{\pr_1^* \smile  \pr_2^*}  H^*_{\et}(X \times X)
\]
(morally, ``$\pr_1 =   \varphi \circ p_1'$'') we have 
\begin{equation}\label{eq: step 3 factor push}
(p_1)_* = (p_1')_* \varphi_*. 
\end{equation}

Now, let us summarize where we are. Combining Lemma \ref{lem: step 2} and  \eqref{eq: step 3 factor push}, we know that
\begin{equation}\label{eq: step 3 eqn2}
w = (p_1)_* \Sq \Delta = (p_1')_* \varphi_* \Sq \Delta.
\end{equation}
Lemma \ref{Delta_basis} gives us an expression for $\varphi_* \Delta$, hence also $\Sq \varphi_* \Delta$. If we could commute $\varphi_*$ and $\Sq$, then this would give us a formula for $\varphi_* \Sq \Delta$. But although Steenrod squares commute with pullbacks, they do \emph{not} in general commute with pushforwards. This is the key problem (note that the whole issue disappears when one has the K\"{u}nneth formula, as in classical algebraic topology or in algebraic geometry over separably closed fields). To address this issue, in the last step of the proof, we will establish:

\begin{prop}\label{commute}
Let $X$ be a smooth proper variety over a finite field, and $\varphi_*$ be as in \eqref{push_end}. Then we have $\Sq \varphi_* = \varphi_* \Sq$.
\end{prop}

%\begin{remark}
%In fact we do not need the smoothness or properness hypotheses. They are included only because at this point we have only defined the map $\varphi_*$ under those hypotheses, but in the next section we will see a way of generalizing this definition. On the other hand, we do not need the result beyond the smooth and proper case. 
%\end{remark}

\emph{Assuming Proposition \ref{commute} for now} (it will be shown in \S \ref{battlefield}), we complete the rest of the proof of Theorem \ref{Wu theorem finite field}. Let $\{e_i\}$ be a basis for $H^*_{\et}(X)$ and $\{f_i\}$ the dual basis under Poincar\'{e} duality, as above. By \eqref{eq: step 3 eqn2}, Proposition \ref{commute}, and Lemma \ref{Delta_basis} we have 
\begin{equation}\label{eq: step 3 eqn3}
w = (p_1')_* \Sq \left(\sum_i  (p_1')^* e_i \smile  (p_2')^* f_i \right).
\end{equation}
By the Cartan formula and the projection formula (which hold by Remark \ref{rem: fake variety}), we have
\begin{align}\label{eq: step 3 eqn4}
(p_1')_* \Sq \left(\sum_i  (p_1')^* e_i \smile  (p_2')^* f_i \right) &=   \sum_i (p_1')_* \left( (p_1')^* \Sq e_i \smile (p_2')^* \Sq f_i \right) \nonumber\\
&= \sum_i  \Sq e_i \smile  (p_1')_*(p_2')^* \Sq f_i \nonumber \\
&= \sum_i \Sq e_i \otimes (p_1')_* (p_2')^* \Sq f_i .
\end{align}

Now, unraveling the definitions shows that
\[
(p_1')_* (p_2')^*  \gamma =  \int \gamma  \quad \text{for all $ \gamma \in H^*_{\et}(X)$}
\]
where the right hand side is viewed in $\Z/2\Z \cong H^0_{\et}(X)$. (It is also easy to see directly that this must be the case for degree reasons, since the left side can only be non-zero for $\gamma$ in top degree.) Combining this with \eqref{eq: step 3 eqn3} and \eqref{eq: step 3 eqn4}, we find that 
\begin{align*}
w &= \sum_i \Sq e_i  \cdot \int \Sq f_i = \sum_i \Sq e_i \cdot \int v \smile f_i  = \Sq  \left(\sum_i e_i \cdot \int v \smile f_i \right) = \Sq v,
\end{align*}
with the last equality using that $\{e_i\}$ and $\{f_i\}$ are dual bases. \qed

\subsubsection{Step 4}\label{battlefield}
This step is devoted to the proof of Proposition \ref{commute}. As foreshadowed in Remark \ref{rem: fake variety}, the difficulty stems from the inability to realize $H^*_{\et}(X) \otimes H^*_{\et}(X)$ as  the cohomology of an actual variety over $\F_q$. For this reason it is useful to pass to \'{e}tale topological type, where we can interpret
\[
H^*_{\et}(X) \otimes H^*_{\et}(X) = H^*_{\et}(\Et(X)\times \Et(X)).
\]

\noindent \textbf{The idea of the argument.} The basic geometric idea is that the map 
\begin{equation}\label{battle eq 4}
\varphi_* \co  H^*_{\et}(X \times X )  \rightarrow H^*_{\et}(X) \otimes H^*_{\et}(X)
\end{equation}
looks like a pushforward map on cohomology induced by a  ``homotopy quotient by $\wh{\Z}$'' at the level of geometric objects. Proposition \ref{commute} is then motivated by the well-known fact (which we prove below in Proposition \ref{Z space version}) that Steenrod operations commute with pushforward through a homotopy quotient by $\Z$, and that cohomologically (with finite coefficients) homotopy quotients by $\Z$ and by $\wh{\Z}$ look the same.

To see why the ``basic geometric idea'' should be true, our heuristic is that for any field $k$ and $G_k := \Gal(\ol{k}/k)$, we should have 
\begin{equation}\label{hquot_hope}
``\Et(X)  = \Et(X_{\ol{k}})_{hG_k}".
\end{equation}
Here if a group $G$ acts on a space $Y$ then we write $Y_{hG} := (Y \times EG)/G$ for the ``homotopy quotient of $Y$ by $G$'', where $EG$ is some contractible space with a free $G$ action, and the quotient is for the (free) diagonal action. The heuristic \eqref{hquot_hope} then suggests that
\begin{align*}
\Et(X \times_{k} X)  &\cong \Et((X \times_{k} X)_{\ol{k}})_{hG_k}  \cong \Et(X_{\ol{k}} \times_{\ol{k}} X_{\ol{k}})_{hG_k}  \cong (\Et(X_{\ol{k}}) \times \Et(X_{\ol{k}}))_{hG_k}
\end{align*}
where the quotient is for the action of the diagonal $G_k$ (leaving a residual action of $G_k$), while 
\begin{align*}
\Et(X) \times \Et(X)  \cong \Et(X_{\ol{k}})_{hG_k} \times \Et(X_{\ol{k}})_{hG_k}  \cong  (\Et(X_{\ol{k}}) \times \Et(X_{\ol{k}}))_{h(G_k \times G_k)}.
\end{align*}
Hence we would have a homotopy quotient
\[
\Et(X \times_{k}X) \rightarrow \Et(X \times_{k}X) _{hG_k} \cong  \Et(X) \times \Et(X)
\]
whose induced pushforward on cohomology agrees recovers \eqref{battle eq 4}. 

In a previous version of this paper, we used some complicated gymnastics in cohomology as a substitute for the fact that we did not know how to rigorously formulate \eqref{hquot_hope}. We are very grateful to an anonymous referee for informing us that there already exists a framework to handle this sort of issue, namely the ``relative \'{e}tale homotopy theory'' developed in \cite{HS}, \cite{BS}. This formalism makes the argument much more efficient and transparent, so we review it next. \\

\noindent \textbf{Relative \'{e}tale homotopy theory.} If $X$ is a variety over a field $k$ and $L/k$ is a finite extension, then $\Et(X_L)$ is a pro-object in simplicial sets equipped with an action of $\Gal(L/k)$ \emph{as a pro-object}. However, it would be better to work with an object that enjoys a level-wise action, rather than an action as a pro-object. In \cite{HS} Harpaz-Schlank defined a refined variant of the \'{e}tale homotopy type denoted $\Et_{/k}(X)$, called the \emph{relative \'{e}tale homotopy type} \cite[\S 9.2.3]{HS}\footnote{Note that the ArXiv version of \cite{HS} is numbered differently from the published version.} of $X$, which is a pro-object in the homotopy category of $G_k$-spaces (which means by definition that every simplex has open stabilizer)\footnote{As in \S \ref{subsec: etale homotopy}, we will refer to simplicial sets as ``spaces'' to make the exposition smoother.}. The point is that $\Et_{/k}(X)$ is equipped with a \emph{level-wise} action of $G_{k}$. In \cite{BS} Barnea-Schlank lifted this construction to the pro-category of $G_k$-spaces, and it is this refinement that we will use in this paper. The improvement is analogous to Friedlander's refinement of the Artin-Mazur \'{e}tale homotopy type, which is a pro-object in the homotopy category of spaces, to the \'{e}tale topological type discussed in \S \ref{subsec: etale homotopy}, which is a pro-object in spaces, although the methods of \cite{BS} are very different. 

Here is a very brief summary of the difference between $\Et_{/k}(X)$ and $\Et(X)$; see \cite[\S 9.2.3]{HS} and \cite[\S 8]{BS} for the details. The basic idea is that the usual definition of $\Et(X)$ attaches to each hypercovering $U_.$ of $X$ the simplicial set $\pi_0(U_.)$ of its connected components. On the other hand, $\Et_{/k}(X)$ assigns to $U_.$ the simplicial set $\pi_0(U_. \times_k\ol{k})$, which is equipped with an obvious $G_k$-action.

\begin{defn}[{\cite[\S 9.6.2]{HS}}]\label{def: htpy quot}
Given a $G_k$-space $Y$, we define the \emph{homotopy quotient} $Y_{hG_k}$ to be the pro-space 
\[
Y_{hG_k} := \{(Y \times E(G_k/H))/G_k\}_{H}
\]
where the index set runs over $i$ and open normal subgroups $H \triangleleft G_k$. 

Given a pro-$G_k$-space $\{Y_i\}_i$, we define the \emph{homotopy quotient} $(\{Y_i\}_{i})_{hG_k}$ to be the pro-space 
\[
(\{Y_i\}_i)_{hG_k} := \{(Y_i\times E(G_k/H))/G_k\}_{i,H}.
\]
where the index set runs over open normal subgroup $H \triangleleft G_k$. 
\end{defn}

\begin{remark}\label{rem: htpy quotient cofinal}
Note that $(Y\times E(G_k/H))/G_k = (Y/H)_{h(G_k/H)}$. In particular, if $H$ acts trivially on $Y$ then $(Y\times E(G_k/H))/G_k =Y_{h(G_k/H)}$.
\end{remark}

We next begin discussing the key properties of $\Et_{/k}(X)$. Actually, we replace $\Et_{/k}(X)$ by its Postnikov truncation denoted $\Et_{/k}(X)^{\sharp}$ in \cite{HS}, which does not alter the cohomology. This is a technical device to guarantee certain finiteness conditions levelwise; for simplicity of notation we will omit the $\sharp$. 
\begin{enumerate}
\item By \cite[Proposition 9.82]{HS}, we have a homotopy equivalence
\[
\Et_{/k}(X)_{hG_k} \xrightarrow{\sim} \Et(X).
\]
\item By the cofinality of the diagonal in the product of a left filtered index category with itself, \cite[Proposition 9.82]{HS}  also implies that we have a homotopy equivalence
\[
(\Et_{/k}(X) \times \Et_{/k}(X))_{hG_k \times hG_k} \xrightarrow{\sim} \Et(X) \times \Et(X).
\]
\item Finally, \cite[Proposition 9.19]{HS} asserts that the underlying pro-space of $\Et_{/k}(X)$ (obtained by forgetting the $G_k$-action) is homotopy equivalent to $\Et(X_{\ol{k}})$. Denote this forgetful functor by $\mrm{Oblv}$.
\end{enumerate}

Now, $\Et_{/k}(X) \times \Et_{/k}(X)$ is naturally a pro-$G_k \times G_k$-space. Taking the homotopy quotient for the diagonal $G_k$-action leaves a residual $G_k$-action, making $(\Et_{/k}(X) \times \Et_{/k}(X))_{hG_k}$ a pro-$G_k$-space, so that we can take the homotopy quotient again. Putting the facts (1)-(3) together, we realize the homotopy quotient map \eqref{hquot_hope} as the vertical map between pro-spaces in the diagram below:
\begin{equation}\label{eq: homotopy quotient diagram}
\begin{tikzcd}
\mrm{Oblv} ((\Et_{/k}(X) \times \Et_{/k}(X))_{hG_k}) \ar[d, "/hG_k"]  \ar[r, "\sim" ] &  (\Et_{/k}(X \times_k X))_{hG_k}  \ar[r, "\sim"] & \Et(X \times_k X)  \\
 (\Et_{/k}(X) \times \Et_{/k}(X))_{h(G_k \times G_k)} \ar[rr, "\sim"]   & & \Et(X) \times \Et(X)
\end{tikzcd} 
\end{equation} 

\noindent \textbf{Comparing homotopy quotients by $\Z$ and $\wh{\Z}$.} Let $Y := \{Y_i\}_i$ be any pro-$\wh{\Z}$-space. We can restrict the $\wh{\Z}$-action to a $\Z$-action, and then form $Y_{h\Z}$. By Definition \ref{def: htpy quot} and Remark \ref{rem: htpy quotient cofinal}, we have $Y_{h\Z} = \{(Y_i)_{h\Z}\}_i$ while  $Y_{h\wh{\Z}} =\{(Y_i/n\Z)_{h(\Z/n\Z)}\}_{i,n}$. Thus there is a canonical map of pro-$\G_k$-spaces:
\begin{equation}\label{eq: hZ to hat}
Y_{h\Z} \rightarrow Y_{h\wh{\Z}}.
\end{equation}
Our next goal is to show that if $Y = \Et_{/k}(X)$ arises as the relative homotopy type of a variety $X/\F_q$, then \eqref{eq: hZ to hat} induces an isomorphism on cohomology. 

\begin{lemma}\label{battle Z vs Z/n} 
Let $X$ be a variety over $\F_q$. Then the natural map
\[
 H^*( \Et_{/k}(X)_{h\Z}; \Z/2\Z)  \leftarrow  H^*( \Et_{/k}(X)_{h\wh{\Z}}; \Z/2\Z)
\]
is an isomorphism.
\end{lemma}

\begin{proof} 
We consider the pro $G_k$-space $Y = (Y_i) := \Et_{/k}(X)$. The map on cohomology induced by \eqref{eq: hZ to hat} is 
\[
\varinjlim_i H^*((Y_i)_{h\Z} ) \leftarrow \varinjlim_i \varinjlim_n H^*((Y_i)_{h(\Z/n\Z)}).
\]
Therefore it will certainly suffice to prove that 
\begin{equation}\label{eq: isom in limit}
H^*((Y_i)_{h\Z} ) \xleftarrow{\sim} \varinjlim_n H^*((Y_i/n\Z)_{h(\Z/n\Z)})
\end{equation}
individually for each $i$. Moreover, since we assume that $Y = \Et_{/k}(X)$ we can take $Y_i$ to be \emph{excellent}, i.e. we can assume that the $\wh{\Z}$-action on $Y_i$ already factors through a finite quotient \cite[\S 9.2.3, p.296-297]{HS}. Hence, by restricting to a cofinal subcategory of the indexing category, and renaming $Y_i$ to $Y$, we may assume that $Y$ is a $\wh{\Z}$-space (as opposed to pro-space) on which the action factors through $\Z/n\Z$. Then, obviously, for any $n \mid N$ we have  
\[
(Y/N\Z)_{h(\Z/N\Z)}  = Y_{h(\Z/N\Z)}.
\]

The identity map $Y \rightarrow Y$ is obviously equivariant for the group homomorphism $\Z \rightarrow \Z/N\Z$, and induces the map of spectral sequences
\begin{equation}
\xymatrix{
H^r(\Z; H^s(Y)) \ar@{=>}[r]  & H^{r+s}(Y_{h\Z}) \\
H^r(\Z/N\Z; H^s(Y)) \ar@{=>}[r] \ar[u] & H^{r+s}(Y_{h(\Z/N\Z)}) \ar[u] \\
}
\end{equation}
This map of spectral sequences is not necessarily an isomorphism for any fixed $N$, but after taking the direct limit as $N$ runs over all positive multiples of $n$, it becomes an isomorphism by the well-known comparison of group cohomology for $\wh{\Z}$ and $\Z$ with finite coefficients:
\begin{equation}\label{eq: big ss}
\begin{tikzcd}
H^r(\Z; H^s(Y)) \ar[r, equals] & H^r(\Z; H^s(Y)) \ar[r, Rightarrow]  & H^{r+s}(Y_{h\Z}) \\
H^r(\wh{\Z}; H^s(Y)) \ar[r, equals] \ar[u, "\sim"] & \varinjlim_{n \mid N} H^r(\Z/N\Z; H^s(Y)) \ar[r, Rightarrow] \ar[u, "\sim"] & H^{r+s}(Y_{h\wh{\Z}}) \ar[u] 
\end{tikzcd}
\end{equation}
Now \eqref{eq: big ss} implies that the natural map in \eqref{eq: isom in limit} induces an isomorphism on associated gradeds, hence an isomorphism. 
\end{proof}

\noindent \textbf{More on pushforwards in cohomology.} We now turn towards studying the pushforward on cohomology. We begin by formalizing the construction of the pushforward in \S \ref{step2}. To help distinguish varieties from pro-spaces, we will use boldface letters to denote pro-spaces in this part. 

\begin{defn}
We say that a pro-space $\mbf{Y}$ has \emph{Poincar\'{e} duality} if there exists an $n$ and an isomorphism $\int \colon H^n(\mbf{Y}; \Z/2\Z) \cong \Z/2\Z$ such that the cup product induces a perfect pairing 
\[
H^i(\mbf{Y};\Z/2\Z) \times H^{n-i}(\mbf{Y};\Z/2\Z) \xrightarrow{\smile} H^n(\mbf{Y};\Z/2\Z) \xrightarrow{\int} \Z/2\Z.
\]
(This uniquely determines $n$.) We denote the non-zero element in $H^n(\mbf{Y}; \Z/2\Z)$ by $\mu_{\mbf{Y}}$, and call it the \emph{fundamental class of $\mbf{Y}$}. From $\int$ we obtain a functional 
\[
H^*(\mbf{Y}; \Z/2\Z) \xrightarrow{\mrm{project}} H^n(\mbf{Y}; \Z/2\Z)  \xrightarrow{\int} \Z/2\Z
\]
which we will also denote by $\int$. 
\end{defn}

\begin{defn}\label{push_by_duality}
Let $f \colon \mbf{W} \rightarrow \mbf{Y}$ be a map of pro-spaces with Poincar\'{e} duality, with fundamental classes $\mu_{\mbf{W}} \in H^m(\mbf{W})$ and $\mu_{\mbf{Y}} \in H^n(\mbf{Y})$. Then we have a pullback map on cohomology 
\[
f^* \colon H^*(\mbf{Y})  \rightarrow H^*(\mbf{W}) 
\]
which induces a dual map
\[
(f^*)^{\vee} \colon H^*(\mbf{W})^{\vee} \rightarrow H^*(\mbf{Y})^{\vee}.
\]
We define the  \emph{pushforward on cohomology} 
\[
f_* \colon H^i(\mbf{W}) \rightarrow H^{i+n-m}(\mbf{Y})
\]
to be the map $(f^*)^{\vee} $, where the identifications 
\[
 H^i(\mbf{W}) = H^{m-i}(\mbf{W})^{\vee} \quad \text{and} \quad  H^{i+n-m}(\mbf{Y})=	 H^{m-i}(\mbf{Y})^{\vee}
\]
are induced by Poincar\'{e} duality. It is easily checked that $f_*$ takes $\mu_{\mbf{W}}  \mapsto \mu_{\mbf{Y}}$. 

\end{defn}

In particular, if $\mbf{W}  =\Et(X)$ for some smooth proper variety $X$ over $\F_q$ or $k^{\mrm{sep}}$, then $\mbf{W}$ inherits Poincar\'{e} duality from $X$. If the map $f\co \mbf{W} \rightarrow \mbf{Y}$ arises from a map of varieties $X \rightarrow V$, then the pushforward $f_*$ tautologically agrees, under the identifications $H^*(\mbf{W}) = H^*_{\et}(X)$ and $H^*(\mbf{Y}) = H^*_{\et}(V)$, with the pushforward we defined in \S \ref{step2}. In addition, $\Et(X) \times  \Et(X)$ has Poincar\'{e} duality by the K\"{u}nneth theorem. Hence, by the horizontal homotopy equivalences in \eqref{eq: homotopy quotient diagram} we find that $(\Et_{/k}(X) \times \Et_{/k}(X))_{hG_k} $ and  $ (\Et_{/k}(X) \times \Et_{/k}(X))_{h(G_k \times G_k)} $ have Poincar\'{e} duality. It is tautological that the diagram 
\[
\xymatrix{
H^*(\Et(X \times_{\F_q} X)  )  \ar@{=}[d]& H^* (\Et(X) \times \Et(X)) \ar@{=}[d] \ar[l] \\
H^*_{\et}(X \times_{\F_q} X) & H^*_{\et}(X) \otimes H^*_{\et}(X) \ar[l]
}
\]
commutes, and dualizing it shows the following Lemma. 

\begin{lemma}\label{lem: push compatible with Et}
Let $X$ be smooth and proper over $\F_q$. Then the map
\[
\varphi_* \co H^*_{\et}(X \otimes_{\F_q} X) \rightarrow H^*_{\et}(X) \otimes H^*_{\et}(X)
\]
defined in \S \ref{step2} coincides with the pushforward in cohomology  (as defined in Definition \ref{push_by_duality}) induced by the vertical map in \eqref{eq: homotopy quotient diagram}
\[
H^*(\mrm{Oblv}(\Et_{/\F_q}(X) \times \Et_{/\F_q}(X))_{h\wh{\Z}}) \rightarrow  H^* ((\Et_{/\F_q}(X) \times \Et_{/\F_q}(X))_{h(\wh{\Z} \times \wh{\Z})})
\]
 under the identifications obtained by \S \ref{subsec: etale homotopy} and the horizontal equivalences in \eqref{eq: homotopy quotient diagram}:
\begin{align*}
H^*(\mrm{Oblv}(\Et_{/\F_q}(X) \times \Et_{/\F_q}(X))_{h\wh{\Z}}) &= H^*_{\et}(X \times_{\F_q} X) , \\
H^* ((\Et_{/\F_q}(X) \times \Et_{/\F_q}(X))_{h(\wh{\Z} \times \wh{\Z})} )&= H^*_{\et}(X) \otimes H^*_{\et}(X).
\end{align*}
\end{lemma}

There are other constructions of the pushforward in cohomology which are more useful for studying the interaction with Steenrod operations, so we prove a general criterion for recognizing when a map on cohomology coincides with the pushforward as we have defined it. 

\begin{lemma}\label{push_criterion}
Suppose $f \colon \mbf{W} \rightarrow \mbf{Y}$ is a map of (pro-)spaces with Poincar\'{e} duality so that $f_*$ is defined as in Definition \ref{push_by_duality}. Keeping the notation of Definition \ref{push_by_duality}, suppose $f'_* \colon H^*(\mbf{W}) \rightarrow H^{*+n-m}(\mbf{Y})$ is another map satisfying:
\begin{enumerate}
\item $f'_*(\mu_{\mbf{W}}) = \mu_{\mbf{Y}}$, and
\item $f'_*(x \smile f^* \gamma) = f'_* x \smile \gamma $ for all $x \in H^*(\mbf{W})$ and $\gamma \in H^*(\mbf{Y})$. 
\end{enumerate}
Then $f'_* = f_*$.
\end{lemma} 

\begin{remark} In other words, Condition (2) is saying that $f'_*$ is a module homomorphism for $H^*(\mbf{Y})$, for its natural action on $H^*(\mbf{Y})$ and its action on $H^*(\mbf{W})$ via $f^*$.
\end{remark}

\begin{proof}
By Poincar\'{e} duality on $\mbf{Y}$, it suffices to show that 
\begin{equation}\label{alt_push_eq1}
\int_{\mbf{Y}} f'_* x \smile \gamma = \int_{\mbf{Y}} f_* x \smile \gamma \text{ for all } x \in H^*(\mbf{W}), \gamma \in H^*(\mbf{Y}).
\end{equation}
By definition (Definition \ref{push_by_duality}), we have
\begin{equation}\label{alt_push_eq2}
\int_{\mbf{Y}} f_* x \smile \gamma = \int_{\mbf{W}} x \smile f^* \gamma,
\end{equation}
Substituting \eqref{alt_push_eq2} into \eqref{alt_push_eq1} and using condition (2), we see that we need to show that 
\begin{equation}\label{alt_push_eq3}
\int_{\mbf{Y}} f'_*(x \smile f^* \gamma) = \int_{\mbf{W}} x \smile f^* \gamma.
\end{equation}
By the definition of the fundamental class, we can write 
\[
x \smile f^* \gamma = \left(\int_{\mbf{W}} x \smile f^* \gamma \right) \mu_{\mbf{W}} + \text{(lower degree terms)},
\]
hence condition (1) of the Lemam implies that 
\begin{equation}\label{alt_push_eq4}
f_*' (x \smile f^* \gamma ) =   \left(\int_{\mbf{W}} x \smile f^* \gamma \right) \mu_{\mbf{Y}} + \text{(lower degree terms)}.
\end{equation}
Substituting \eqref{alt_push_eq4} into \eqref{alt_push_eq3}, we have
\[
\int_{\mbf{Y}} f'_*(x \smile f^* \gamma)  = \left(\int_{\mbf{W}} x \smile f^* \gamma \right) \cdot \left(\int_{\mbf{Y}} f'_* \mu_{\mbf{W}}  \right) = \int_{\mbf{W}} x \smile f^* \gamma 
\]
where the last equality follows from condition (1) that $f'_* \mu_{\mbf{W}} = \mu_{\mbf{Y}}$ and the fact that the lower degree terms ``integrate'' to $0$ by definition.
\end{proof}

We use this discussion to study the pushforward in cohomology induced by a homotopy quotient of \emph{spaces} $M \rightarrow M_{h\Z}$. The following proposition is presumably well-known, but we have included a proof since we could not find a reference. 

\begin{prop}\label{Z space version}
Let $M$ be any simplicial set with $\Z$-action. If  
\[
f \colon M \rightarrow M_{h\Z}
\]
 denotes the homotopy quotient map, then we can define a pushforward map 
 \[
 f_* \co H^i(M)  \rightarrow H^{i+1}(M_{h\Z}).
 \]
If $M$ has Poincar\'{e} duality, then so does $M_{h\Z}$ and $f_*$ agrees with the pushforward defined in Definition \ref{push_by_duality}.
 
 Moreover, for all $x \in H^*(M)$ we have 
\[
\Sq \circ f_* (x) = f_* \circ \Sq (x).
\]
\end{prop}

\begin{proof}
We can take $\R = E\Z$ as a model for $E\Z$, so that a model for $M_{h\Z}$ is $(M \times \R) /\Z$. (The map $M \rightarrow (M \times \R) /\Z$ is what might classically be called the ``inclusion of a fiber into the mapping torus''.) With this model there is an evident homeomorphism
\[
((M \times \R )/\Z,  M) \cong (S^1 \wedge (M_+) , \pt)
\]
where the left side means the pointed space obtained from $(M \times \R )/\Z$ by collapsing $M\times \{0\}$ to a point. 

We always have a pushforward map $H^*(M) \rightarrow H^{*+1}(M_{h\Z})$ defined as the composition
\begin{equation}\label{eq: push alternative}
\begin{tikzcd}
H^*(M) \ar[rr, dashed] \ar[dr, "a"', "\sim"]  & & H^{*+1}(M_{h\Z}) \\
& H^{*+1}(S^1 \wedge (M_+)) \ar[ur, "b"]
\end{tikzcd}
\end{equation}
If $M$ has Poincar\'{e} duality then so does $M_{h\Z}$ by the same reasoning as in \S \ref{subsec: Tate pairing}, so we have fundamental classes $\mu_M$ and $\mu_{M_{h\Z}}$ and Definition \ref{push_by_duality} supplies another notion of pushforward $H^*(M) \rightarrow H^{*+1}(M_{h\Z})$. In this case, the alternate pushforward \eqref{eq: push alternative} evidently takes $\mu_M \mapsto \mu_{M_{h\Z}}$, and is an $H^{*}(M_{h\Z})$-module homomorphism since $b$ is induced by a map of spaces $M_{h\Z} \rightarrow S^1 \wedge (M_+)$. Therefore, Lemma \ref{push_criterion} shows that the two notions of pushforward coincide.	

Finally, the homomorphism $b$ commutes with Steenrod operations because it is induced by the map of spaces, and the map $a$ commutes with Steenrod operations because it is a suspension isomorphism, and Steenrod operations always commute with suspension (\S \ref{subsec: classical Steenrod}).
\end{proof}

\noindent \textbf{The punchline.} We now finally complete the proof of Proposition \ref{commute}. We transfer the Poincar\'{e} duality structure 
\[
\text{from $H^*(\Et(X \times_{\F_q} X))$ to $ H^*((\Et_{/k}(X) \times \Et_{/k}(X))_{h\Z})$ }
\]
and 
\[
\text{from $H^*(  \Et(X) \times \Et(X))$ to $ H^*((\Et_{/k}(X) \times \Et_{/k}(X))_{h(\Z  \times \Z)})$}
\]
using the horizontal isomorphisms in \eqref{eq: homotopy quotient diagram} plus Lemma \ref{battle Z vs Z/n}. Then Definition \ref{push_by_duality} furnishes a notion of pushforward 
\begin{equation}\label{battle eq 7}
H^*((\Et_{/k}(X) \times \Et_{/k}(X))_{h\Z}) \rightarrow   H^*((\Et_{/k}(X) \times \Et_{/k}(X))_{h(\Z  \times \Z)})
\end{equation}
which by Lemma \ref{lem: push compatible with Et} is compatible the map in Proposition \ref{commute} in the sense that the following diagram commutes:
\[
\xymatrix{
H^*((\Et_{/k}(X) \times \Et_{/k}(X))_{h\Z})\ar[r] \ar@{=}[d]  &  H^*((\Et_{/k}(X) \times \Et_{/k}(X))_{h(\Z  \times \Z)})\ar@{=}[d]\\
H_{\et}^*(X \times_{\F_q} X) \ar[r]_{\varphi_*} & H^*_{\et}(X) \otimes H^*_{\et}(X)
}
\]

On the other hand, we may write $(\Et_{/k}(X) \times \Et_{/k}(X))_{h\Z} = \{U_i\}_i$, with each $U_i$ being a $\Z$-space. Then by Lemma \ref{push_criterion} the map \eqref{battle eq 7} agrees with the colimit of the levelwise pushforwards 
\[
H^*(U_i) \rightarrow  H^*((U_i)_{h\Z}),
\]
each of which commutes with Steenrod operations by Lemma \ref{Z space version}. \qed

\section{The obstruction to being alternating}\label{sec: alternating}

\subsection{Lifting Wu classes to integral cohomology}\label{subsec: lifting SW}

The goal of this subsection is to prove that Wu classes lift to integral cohomology, which will turn out to be important later.

\begin{lemma}\label{lem: wu in terms of SW}
Every Wu class $v_{s}$ can be expressed as a polynomial in the Stiefel-Whitney classes $\{w_k\}$. 
\end{lemma}

\begin{proof}
We induct on $s$. The base case is $v_0 = w_0 =1$. Consider the equation 
\[
\Sq v = w
\] 
from Theorem \ref{Wu theorem finite field}. Equating terms in cohomological degree $s$, we have 
\[
v_s + \Sq^1v_{s-1} + \ldots  = w_s.
\]
By the induction hypothesis each term $v_{s-i}$ is a polynomial in the Stiefel-Whitney classes, so by the Cartan formula for Steenrod squares (\S \ref{subsec: classical Steenrod}) and Lemma \ref{lem: squares of SW classes}, each $\Sq^i v_{s-i}$ is a polynomial in the Stiefel-Whitney classes. Then solving for $v_s$ completes the induction.
\end{proof}

\begin{cor}\label{cor: wu class lifts} The Wu class $v_{s} \in H^{s}_{\et}(X;\Z/2\Z)$ is the reduction of a class in $H^{s}_{\et}(X;\Z_2(\lceil s/2 \rceil))$.
\end{cor}

\begin{proof}
By Lemma \ref{lem: wu in terms of SW} we can express  $v_{s}$ as a polynomial $P_{s}(\{w_i\})$ in just the Stiefel-Whitney classes. Using Theorem \ref{thm: chern classes}, rewrite $P_{s}(\{w_i\})$ as a polynomial in the (reductions of) Chern classes and $\alpha$.

Note that the Chern classes all live in even cohomological degree while $\alpha$ has degree 1 and $\alpha^2 = 0$. Therefore, if $s=2k$ is even then this polynomial can be written without $\alpha$, while if $s=2k+1$ is odd then it can be written in the form $\alpha P_{s}'(\{c_i\})$. In any case, Theorem \ref{thm: chern classes} tells us that the Chern classes lift to integral cohomology, and so does $\alpha$ by Remark \ref{rem: alpha}.
\end{proof}

\subsection{Proof of the main theorem}\label{subsec: proof main theorem}

We now combine the preceding ingredients to prove Theorem \ref{thm: pseudo_main}, which as already noted implies Theorem \ref{thm: main}. We wish to show that 
\[
x \smile \beta (x)  = 0  \text{ for all } x \in H^{2d}_{\et}(X;\Z/2^n\Z(d)).
\]

Theorem \ref{thm: pairing identity} tells us that
\begin{equation}\label{eq: final comp 1}
x \smile \beta (x) = \wt{\Sq}^{2d}( \beta(x)).
\end{equation}
Then Lemma \ref{lemma: etale generalized sq} implies that
\begin{equation}\label{eq: final comp 2}
\wt{\Sq}^{2d} (\beta (x)) = [2^{n-1}] \circ \Sq^{2d} (\ol{\beta (x)}),
\end{equation}
where $\ol{\beta (x)}$ denotes the reduction of $\beta(x)$ mod $2$, and the notation $[2^{n-1}]$ is as in Lemma \ref{lemma: etale generalized sq}.

By the definition of the Wu classes in Theorem \ref{Wu theorem finite field}, we have
\begin{equation}\label{eq: final comp 3}
\wt{\Sq}^{2d} (\beta (x)) = [2^{n-1}] ( v_{2d} \smile \ol{\beta (x)} ).
\end{equation}
It is immediate from the definition of $[2^{n-1}]$ that 
\begin{equation}\label{eq: final comp 4}
[2^{n-1}] (v_{2d} \smile \ol{\beta(x)}) = ([2^{n-1}] v_{2d}) \smile \beta(x).
\end{equation}
Next, since $\beta$ is a derivation we have 
\begin{equation}\label{eq: final comp 5}
 ([2^{n-1}] v_{2d})  \smile \beta(x)   = \beta([2^{n-1}] v_{2d} \smile x )  - \beta([2^{n-1}] v_{2d} ) \smile x.
\end{equation}
Stringing together \eqref{eq: final comp 1}, \eqref{eq: final comp 2}, \eqref{eq: final comp 3}, and \eqref{eq: final comp 4}, it suffices to show that \eqref{eq: final comp 5} is $0$ for all $x$. 

\begin{lemma}\label{lem: two obstructions}
We have $\beta([2^{n-1}] v_{2d} ) =  \beta_{2,2^n}(v_{2d})$ where $\beta_{2,2^n}$ is the boundary map for 
\[
0 \rightarrow \Z/2^n \Z(d) \xrightarrow{2}  \Z/2^{n+1} \Z(d)  \rightarrow \Z/2 \Z (d) \rightarrow 0.
\]
\end{lemma}

\begin{proof}
This follows immediately from the commutative diagram:
\[
\begin{tikzcd}
0 \ar[r]  &  \Z/2^n \Z(d) \ar[r, "2"] \ar[d, equals] &   \Z/2^{n+1} \Z(d) \ar[r], \ar[d, "2^{n-1}"] &  \Z/2 \Z(d) \ar[r] \ar[d, "2^{n-1}"]  \ar[r] & 0 \\
0 \ar[r] & \Z/2^n \Z (d) \ar[r, "2^{n}"] & \Z/2^{2n} \Z (d)  \ar[r] & \Z/2^n \Z(d) \ar[r] & 0
\end{tikzcd}
\]
\end{proof}

Referring to \eqref{eq: final comp 5}, we have $\beta([2^{n-1}] v_{2d} \smile x )   = 0$ for all $x$ by Lemma \ref{lem: top boundary vanishes}. But we also have $\beta_{2,2^n}(v_{2d})  = 0$ as this is the obstruction to lifting $v_{2d}$ to $H^{2d}_{\et}(X; \Z/2^{n+1} \Z(d))$, and $v$ even lifts to $v_{2d} \in H^{2d}_{\et}(X; \Z_2(d))$ by Corollary \ref{cor: wu class lifts}. By Lemma \ref{lem: two obstructions} we also get $\beta([2^{n-1}] v_{2d}) = 0$, so the expression in \eqref{eq: final comp 5} vanishes for all $x$, which completes the proof. \qed

\subsection{A criterion for the alternation of the linking form}\label{subsec: alternating linking form}

The argument of \S \ref{subsec: proof main theorem}, and the ingredients going into it, can be adapted in a straightforward manner to study linking form of an odd-dimensional manifold (\S \ref{subsec: analogy linking form}). It yields the following conclusion, which we state in the orientable case for simplicity. 

\begin{thm}\label{thm: alternating linking form}
Let $M$ be an orientable closed manifold of dimension $4d+1$. Then the linking form on $M$ is alternating if and only if $v_{2d}$ lifts to $H^{2d}(M; \Z)$, i.e. if and only if $\wt{\beta} (v_{2d}) = 0$ where $\wt{\beta}$ is the boundary map for the short exact sequence
\[
0 \rightarrow \Z \xrightarrow{2} \Z \rightarrow \Z/2 \Z \rightarrow 0.
\]
\end{thm}

\begin{proof}
The argument in \S \ref{subsec: proof main theorem}, with the appropriate adaptations of its ingredients to singular cohomology, shows that the linking form is alternating if and only if $\beta_{2,2^n}(v_{2d}) = 0$ for all $n$. Since $H^{2d+1}(M; \Z)$ is finitely generated, the reduction maps induce an injection 
\[
H^{2d+1}(M; \Z) \hookrightarrow \varprojlim_n H^{2d+1}(M; \Z/2^n  \Z),
\]
which implies that $\ker \wt{\beta} = \bigcap_n \ker \beta_{2,2^n}$. 
\end{proof}

\begin{remark}
The statement of Theorem \ref{thm: alternating linking form} is still true for a non-orientable manifold as in \S \ref{subsec: analogy linking form}, if the short exact sequence is replaced with its twist by $\Cal{L}$. Moreover, our theory shows that the condition can be reformulated in terms of (twisted) Stiefel-Whitney classes defined analogously to those in  \S \ref{subsec: SW construction}.	
\end{remark}

\begin{example}
Let us specialize  Theorem \ref{thm: alternating linking form} to recover the known criterion, which was stated in \S \ref{subsec: analogy linking form}, for the linking form on an orientable 5-manifold to be alternating. Assume that $M$ is a smooth, orientable 5-manifold. Then Theorem \ref{thm: alternating linking form} tells us that the linking form is alternating if and only if $v_2$ lifts to integral cohomology. By Example \ref{ex: wu classes}, noting that $w_1 = 0$ since $M$ is orientable, we have $v_2 = w_2(M) \in H^2(M ; \Z/2\Z)$. Finally, it is well-known (see for example \cite[Theorem D.2]{LM}) that $M$ is $\mrm{spin}^{c}$ if and only if $w_2(M)$ lifts to integral cohomology. 
\end{example}

\appendix

\numberwithin{equation}{section}

\bibliographystyle{amsalpha}% Use the "unsrtnat" BibTeX style for formatting the Bibliography

\bibliography{Bibliography}

\end{document}